\documentclass{amsart}

\usepackage{graphicx}
\usepackage{amsmath,amsfonts,amsthm,amssymb, amscd}
\usepackage{mathabx,epsfig}
\usepackage{stmaryrd,amssymb}
\usepackage{tikz}
\usepackage{tikz-cd}
\usepackage{mathrsfs}
\usepackage{hyperref}
\usepackage{enumitem}
\hypersetup{
    colorlinks=true,
    linkcolor=black,
    citecolor=black,
    filecolor=black,
    urlcolor=black,}
\usepackage{bbm}
\usepackage[left=1.4in,right=1.4in,bottom=1.25in]{geometry}

\newtheorem{theorem}{Theorem}[section]
\newtheorem{lemma}[theorem]{Lemma}
\newtheorem{prop}[theorem]{Proposition}
\newtheorem{cor}[theorem]{Corollary}
\newtheorem{definition}[theorem]{Definition}

\newtheorem{innercustomthm}{Theorem}
	\newenvironment{customthm}[1]
  		{\renewcommand
		\theinnercustomthm{#1}\innercustomthm}
		{\endinnercustomthm}
  
\newtheorem{innercustomprop}{Proposition}
	\newenvironment{customprop}[1]
	{\renewcommand
	\theinnercustomprop{#1}\innercustomprop}
	{\endinnercustomprop}
  
\newtheorem{innercustomlem}{Lemma}
	\newenvironment{customlem}[1]
	{\renewcommand\theinnercustomlem{#1}\innercustomlem}
	{\endinnercustomlem}

\newcommand*{\dual}[1]{{#1}^{\scalebox{.75}{\!$\boldsymbol{\vee}$}}}

\numberwithin{equation}{section}

\begin{document}

\title{Hida duality and the Iwasawa main conjecture}

\author{Matthew J. Lafferty}
\address{Department of Mathematics, University of Arizona}
\email{mlaffert@email.arizona.edu}






\begin{abstract} \noindent  The central result of this paper is a refinement of Hida's duality theorem between ordinary $\Lambda$-adic modular forms and the universal ordinary Hecke algebra. Specifically, we give a necessary condition for this duality to be integral with respect to particular submodules of the space ordinary $\Lambda$-adic modular forms. This refinement allows us to give a simple proof that the universal ordinary cuspidal Hecke algebra modulo Eisenstein ideal is isomorphic to the Iwasawa algebra modulo an ideal related to the Kubota-Leopoldt $p$-adic $L$-function. The motivation behind these results stems from Ohta's proof of the Iwasawa main conjecture over $\mathbb{Q}$. Specifically, the most general application of this argument, which employs results on congruence modules and requires one to make some restrictive hypotheses. Using our results we are able to extend Ohta's argument and remove these hypotheses.
\end{abstract}

\maketitle

\section{Introduction}

In order to fix notation, we begin by recalling the setup and statement of the Iwasawa main conjecture over $\mathbb{Q}$.  We fix a prime $p\geq 5$ throughout. For an arbitrary Dirichlet character $\varphi$, we let $M_\varphi$ and $f_\varphi$ denote its modulus of definition and conductor, respectively. For a non-negative integer $n$, we denote the primitive character associated to the product $\varphi\omega^{-n}$ by $\varphi_n$, where $\omega$ is the usual Teichm\"{u}ller character. To keep the notation compact and avoid ambiguity we set $\varphi_n^{-1} = (\varphi_n)^{-1}$. Let $\theta$ and $\psi$ be Dirichlet characters (possibly imprimitive) such that $p\nmid M_\psi$, $p^2\nmid M_\theta$, and $\theta\psi$ is even. Define
\begin{align*}
F &~=~\text{abelian extension of $\mathbb{Q}$ corresponding to $\ker(\theta_{p-2})\cap \ker(\psi_0)$,}\\
F_\infty&~=~\text{cyclotomic $\mathbb{Z}_p$-extension of $F$,}\\
H_\infty&~=~\text{maximal unramified pro-$p$ abelian extension of $F_\infty$.}
\end{align*}

\noindent The group $\mathrm{Gal}(F_\infty/\mathbb{Q})$ acts on $X_\infty:=\mathrm{Gal}(H_\infty/F_\infty)$ via conjugation, and our assumptions on $M_\theta$ and $M_\psi$ imply that $\mathrm{Gal}(F_\infty/\mathbb{Q})\cong\Delta\times \Gamma$, where $\Delta:=\mathrm{Gal}(F/\mathbb{Q})$ and $\Gamma:=\mathrm{Gal}(F_\infty/F)\cong\mathbb{Z}_p$, making $X_\infty$ a $\mathbb{Z}_p[\Delta]\llbracket \Gamma\rrbracket$-module.

Set $\xi= (\theta^{-1}\psi)_0$ and let $\mathcal{O}_\xi=\mathbb{Z}_p[\xi]$ denote the ring generated over $\mathbb{Z}_p$ by the values of $\xi$. We remark that our choice of $\xi$ is inverse to that of Ohta. Define $$X_{\infty,\xi_1} ~=~X_\infty\otimes_{\mathbb{Z}_p[\Delta]} \mathcal{O}_\xi,$$ 

\noindent where the homomorphism $\mathbb{Z}_p[\Delta]\rightarrow \mathcal{O}_\xi$ is induced by $\xi_1$.  It can be shown that $X_{\infty,\xi_1}$ is a finitely generated torsion $\mathcal{O}_\xi\llbracket\Gamma\rrbracket$-module, and the structure of such modules is well understood. Let $\gamma$ be a topological generator of $\Gamma$, and identify the Iwasawa algebra $\mathcal{O}_\xi\llbracket\Gamma\rrbracket$ with $\Lambda_\xi:=\mathcal{O}_\xi\llbracket X\rrbracket$ through the continuous $\mathcal{O}_\xi$-linear map induced by $\gamma \mapsto 1+X$. We then have a homomorphism $$X_{\infty,\xi_1}\rightarrow \Lambda_\xi/(f_1)\oplus \cdots \oplus \Lambda_\xi/(f_r)$$ 

\noindent with finite kernel and cokernel, where the $f_i$ are non-zero elements of $\Lambda_\xi$. We refer to such a homomorphism as a pseudo-isomorphism. While the $f_i$ are not uniquely determined by $X_{\infty,\xi_1}$, their product is. The ideal $(f_1\cdots f_r)\subset\Lambda_\xi$ is called the characteristic ideal of $X_{\infty,\xi_1}$, which we denote by $\mathrm{Char}_{\Lambda_\xi}(X_{\infty,\xi_1}).$

Next, we recall that if $\varphi$ is a Dirichlet character with conductor not divisible by $p^2$, then there exists a unique element \begin{center}
$F(X,\varphi) \in \left\{\begin{array}{rcl}\mathbb{Z}_p[\varphi]\llbracket X\rrbracket & & \varphi\neq \mathbbm{1}\\
\frac{1}{X-p}\mathbb{Z}_p\llbracket X\rrbracket & & \varphi=\mathbbm{1},\end{array}\right.$\end{center} where $\mathbbm{1}$ denotes the trivial character modulo $1$, such that for all integers $k\geq 2$ and $\overline{\mathbb{Q}}^\times$-valued characters $\epsilon$ on $1+p\mathbb{Z}_p$ having $p$-power order, $$F(\epsilon(u)u^{k-2}-1,\varphi) ~=~ L_p(k-2,\varphi\epsilon^{-1}),$$ with $u:=1+p$ and $L_p(s,\varphi\epsilon^{-1})$ denoting the Kubota-Leopoldt $p$-adic $L$-function associated to the character $\varphi\epsilon^{-1}$ \cite[Theorem. 7.10]{Wash}.

\begin{theorem}[Iwasawa main conjecture over $\mathbb{Q}$] We have the following equality of ideals, $$\mathrm{Char}_{\Lambda_\xi}(X_{\infty,\xi_1})~=~\big(F(X,\xi_2^{-1})\big).$$\end{theorem}

The main conjecture was first proven by Mazur and Wiles \cite{MW}, and has since been proven in even greater generality. In \cite{WilesTR},  Wiles simplified the proof in his paper with Mazur while generalizing it to extensions of totally real fields, including the case when $p=2$.  Around the same time, Rubin gave a much simpler proof of the main conjecture over $\mathbb{Q}$ (resp., over imaginary quadratic fields) using a tool from Galois cohomology known as an Euler system \cite{Rubin}. More recently, Ohta has given a simple proof in the spirit of Mazur and Wiles \cite{OhtaEC1, OhtaEC2, OhtaCM, OhtaCF1, OhtaCF2}. The most general application of Ohta's argument \cite{OhtaCM} employs results on congruence modules which require the following hypotheses:
\begin{center}
\begin{enumerate}
\itemsep0em
\item[(H1)] $p\nmid \varphi(N)$ ($\varphi$ is Euler's totient function)\label{euler},

\item[(H2)] the pair $(\theta,\psi)$ is \emph{non-exceptional}: $(\theta\psi^{-1})_{p-2}(p)\neq 1$.
\end{enumerate}
\end{center}

\noindent The results in this paper were obtained in an effort to extend Ohta's proof in \cite{OhtaCM} by circumventing the obstructions arising from his congruence module argument.

Before giving a summary of the main results that will be used to extend Ohta's proof, we want to briefly comment on why one might be interested in doing so. Recently, Sharifi \cite{Sh} has conjectured a deep relationship between $X_\infty$ and $p$-adic Eichler-Shimura cohomology groups of modular curves, and one can show that the Iwasawa main conjecture over $\mathbb{Q}$ is a shadow of this deeper relationship \cite{FKS}. However, Sharifi's constructions incorporate Ohta's work on the main conjecture, and as such the above hypotheses are assumed. By removing these hypotheses in the context of Ohta's proof of the main conjecture, one hopes to be able to free Sharifi's conjectures of them as well.

\subsection{Main results}  We fix algebraic closures $\overline{\mathbb{Q}}$ and $\overline{\mathbb{Q}}_p$, and  denote the completion of $\overline{\mathbb{Q}}_p$ by $\mathbb{C}_p$. Throughout this paper $\mathcal{O}$ will denote the ring of integers of a complete subfield of $\mathbb{C}_p$ with uniformizer $\pi$. We set $\Lambda=\mathcal{O}\llbracket X\rrbracket$. Fix a positive integer $N$ coprime to $p$ satisfying $M_\theta M_\psi \mid Np$. Let $M_\Lambda$ (resp., $S_\Lambda$) denote the space of ordinary $\Lambda$-adic modular forms (resp., ordinary $\Lambda$-adic cusp forms) of level $N$, and $\mathfrak{H}_\Lambda$ (resp., $\mathfrak{h}_\Lambda$) Hida's universal ordinary Hecke  algebra (resp., universal ordinary cuspidal Hecke algebra) of level $N$. It is well known that $M_\Lambda$, $S_\Lambda$, $\mathfrak{H}_\Lambda$, and $\mathfrak{h}_\Lambda$ are free and finitely generated $\Lambda$-modules.

For the moment, let us assume that $\mathcal{O}$ contains \emph{all} roots of unity. In \cite{OhtaCM}, Ohta considers the following exact sequence of $\mathfrak{H}_\Lambda$-modules, $$\begin{CD}0 @>>> S_\Lambda @>>> M_\Lambda @>\mathrm{Res}_\Lambda>> C_\Lambda @>>> 0,\end{CD}$$

\noindent where $\mathrm{Res}_\Lambda$ is the $\Lambda$-adic residue map and $C_\Lambda$ is the space of ordinary $\Lambda$-adic cusps. Of particular interest is the image of ordinary $\Lambda$-adic Eisenstein series under $\mathrm{Res}_\Lambda$, as understanding this image allows one to determine congruences between such series and ordinary $\Lambda$-adic cusp forms. For an Eisenstein series $\mathcal{E}\in M_\Lambda$ associated to a pair of primitive, non-exceptional characters, Ohta was able to determine $\mathrm{Res}_\Lambda(\mathcal{E})$ by localizing the above sequence at the Eisenstein maximal ideal $\mathfrak{M}:=(\pi,X,\mathrm{Ann}_{\mathfrak{H}_\Lambda}(\mathcal{E}))\subset \mathfrak{H}_\Lambda$. Specifically, he was able to isolate the image of $\mathcal{E}$ under the $\Lambda$-adic residue map by showing that the localization $C_{\Lambda,\mathfrak{M}}$ is a free $\Lambda$-module of rank 1. Unfortunately, this argument cannot be extended to Eisenstein series associated to pairs of exceptional characters, as $C_{\Lambda,\mathfrak{M}}$ in this case is free of rank 2. In this paper, rather than localizing the above sequence at $\mathfrak{M}$, we compute the image of Eisenstein series associated to arbitrary pairs of characters under the $\Lambda$-adic residue map directly.

\begin{customthm}{\ref{jun22-2015-947am}} Suppose $t$ is a positive integer coprime to $p$ satisfying $M_\theta M_\psi t\mid Np$, and let $\mathcal{E}_{\theta,\psi;t}\in M_\Lambda$ denote the Eisenstein series associated to the tuple $(\theta,\psi,t)$. Then $$\mathrm{Res}_\Lambda(\mathcal{E}_{\theta,\psi;t}) ~=~ A_{\theta,\psi}\cdot \mathfrak{e}_{\theta,\psi;t}$$ for an explicitly determined $\mathfrak{e}_{\theta,\psi;t}\in C_{\Lambda}$, with $$A_{\theta,\psi}:= \delta_{\theta,\psi}(X)\left(\prod_{\substack{\ell\mid f_\theta f_\psi \\ \ell\nmid f_{\xi}}} ((1+X)^{s(\ell)} - \xi(\ell)\ell^{-2})\right)F(u^{-1}(1+X)^{-1} - 1,\xi_2^{-1})~\in~\Lambda_\xi$$ where $$\delta_{\theta,\psi}(X) ~=~ \left\{\begin{array}{ccl}u^{-1}(1+X)^{-1}-u^{-2} & & (\theta_0,\psi_0)=(\omega^{-2},\mathbbm{1})\\ 1 & & (\theta_0,\psi_0)\neq(\omega^{-2},\mathbbm{1}).\end{array}\right.$$ Furthermore, $\mathfrak{e}_{\theta,\psi;t}\not\in \mathfrak{m} C_{\Lambda}$, where $\mathfrak{m}$ denotes the maximal ideal of $\Lambda$.\end{customthm}

\noindent We note that when $(\theta_0,\psi_0) = (\omega^{-2},\mathbbm{1})$, we have $A_{\theta,\psi}\in\mathbb{Z}_p\llbracket X\rrbracket^\times$  \cite[Lemma 7.12]{Wash}.

Let us now ease our restriction on $\mathcal{O}$, and assume only that $\mathcal{O}$ contains the values of $\theta$ and $\psi$. Using the above theorem we are able to construct a canonical element $\mathcal{F}_{\theta,\psi;t}\in M_\Lambda$ that maps to $\mathfrak{e}_{\theta,\psi;t}\in C_\Lambda$ under the $\Lambda$-adic residue map. This form arises from congruences between $\mathcal{E}_{\theta,\psi;t}$ and ordinary $\Lambda$-adic cusp forms, and has the following nice properties:

\begin{enumerate}

\item $\mathcal{F}_{\theta,\psi;t}\not\in\mathfrak{m}M_\Lambda$.

\item When $\psi=\mathbbm{1}$, we have $a_0(\mathcal{F}_{\theta,\psi;t})\in \Lambda^\times$.

\item $\mathcal{F}_{\theta,\psi;t}$ is a Hecke eigenform modulo $S_\Lambda$, whose eigenvalues agree with those of $\mathcal{E}_{\theta,\psi;t}$.

\end{enumerate}

\noindent Using the above properties and the congruences that define $\mathcal{F}_{\theta,\psi;t}$, we are able to prove the following refinement of Hida's duality theorem.

\begin{customthm}{\ref{jun22-2015-148pm}} Let $\mathcal{V}$ be a free $\Lambda$-submodule of $M_\Lambda$ that contains $S_\Lambda$ and is stable under the action of $\mathfrak{H}_\Lambda$. Denote the quotient field of $\Lambda$ by $Q(\Lambda)$ and define $$\mathcal{V}_0 ~=~ \{F\in \mathcal{V}\otimes_\Lambda Q(\Lambda) : a_n(F)\in \Lambda~\text{for~all~}n\geq 1\}.$$

\noindent Because $\mathcal{V}$ is stable under the action of $\mathfrak{H}_\Lambda$, we know that $$\mathcal{V}\otimes_\Lambda Q(\Lambda) =\langle \mathcal{E}_{\theta_1,\psi_1;t_1},\dots,\mathcal{E}_{\theta_m,\psi_m;t_m},F_1,\dots,F_s\rangle_{Q(\Lambda)}$$ where the tuples $(\theta_i,\psi_i;t_i)$ are distinct and $\{F_1,\dots,F_s\}$ is a $\Lambda$-basis of $S_\Lambda$
 $($Here we are assuming that $\mathcal{O}$ contains the values of all $\theta_i$ and $\psi_i)$. Define $\mathfrak{H}(\mathcal{V})$ to be the $\Lambda$-subalgebra of $\mathrm{End}_\Lambda(\mathcal{V})$ generated by the Hecke operators $\{T_n:n\geq 1\}$.

If the following conditions are satisfied for all integers $i$ and $j$ with $1\leq i<j\leq m$:
\begin{enumerate}\itemsep0em\item[$(i)$] $(\theta_i)_0\nequiv (\theta_j)_0~(\mathrm{mod}\,\pi) \text{ or~~}(\psi_i)_0\nequiv (\psi_j)_0~(\mathrm{mod}\,\pi),$

\item[$(ii)$] $(\theta_i)_0\nequiv (\psi_j\omega^{-1})_0~(\mathrm{mod}\,\pi) \text{ or~~}(\psi_i)_0\nequiv (\theta_j\omega)_0~(\mathrm{mod}\,\pi).$\end{enumerate}

\noindent we have $\mathcal{V}_0 = \langle \mathcal{F}_{\theta_1,\psi_1;t_1},\dots,\mathcal{F}_{\theta_m,\psi_m;t_m},F_1,\dots,F_s\rangle_{\Lambda} \subset M_\Lambda$ and the pairing $$\mathcal{V}_0\times \mathfrak{H}(\mathcal{V})\rightarrow \Lambda:(F,H)\mapsto a_1(F|H)$$ is perfect.
\end{customthm}

For the remainder of this subsection, let us assume that $M_\theta M_\psi =N$ or $Np$ and $\mathcal{O} = \mathbb{Z}_p[\theta,\psi]$. As we will see later, the former assumption ensures that $\mathcal{E}_{\theta,\psi;1}$ is a normalized common eigenform for $\mathfrak{H}_\Lambda$. Using the above refinement of Hida's duality theorem we are able to give a simple proof of the following proposition. 

\begin{customprop}{\ref{sep10-2014-304pm}} Let $I_{\theta,\psi}$ denote the image of $\mathrm{Ann}_{\mathfrak{H}_\Lambda}(\mathcal{E}_{\theta,\psi;1})$ in $\mathfrak{h}_\Lambda$. Then we have the following isomorphism of ${\Lambda}$-algebras $$\mathfrak{h}_\Lambda/I_{\theta,\psi}~\cong~ {\Lambda}/(A_{\theta,\psi}).$$ \end{customprop}

The form of this result is well known. It was first proven by Mazur and Wiles in the case when $\psi=\mathbbm{1}$ and $\theta$ is primitive and non-exceptional \cite{MW}.  In \cite{OhtaCM}, Ohta removed the triviality condition on $\psi$, proving the result for pairs of primtive, non-exceptional characters. Unfortunately, his proof requires the Iwasawa main conjecture over $\mathbb{Q}$. Using Katz's p-adic modular forms, Emerton has given a proof of the above isomorphism in the case when $\psi=\mathbbm{1}$ and $\theta$ is a nontrivial power of the Teichm\"{u}ller character \cite{Em}. In fact, his method was the inspiration for the proof of Theorem \ref{jun22-2015-148pm}. The novelty of our approach lies in its simplicity and generality. Specifically, our proof does not require the Iwasawa main conjecture over $\mathbb{Q}$ and makes no restrictions on the characters $\theta$ and $\psi$ apart from those required in the definition of the ordinary $\Lambda$-adic Eisenstein series $\mathcal{E}_{\theta,\psi;1}$.

With the above results in hand we are able to extend Ohta's proof of the main conjecture. Let us give a brief overview of how we will go about doing so. For reasons that will be made clear later, it suffices to construct an unramified pro-$p$ abelian extension $L_\infty$ of $F_\infty$ satisfying the following conditions: 

\begin{enumerate}
\itemsep0em
\item[(H1)] $\Delta$ acts on $\mathrm{Gal}(L_\infty/F_\infty)$ via $\xi_1$,

\item[(H2)] $\mathrm{Char}_{\Lambda_\xi}(\mathrm{Gal}(L_\infty/F_\infty))=\big(F(X,\xi_2^{-1})\big).$

\end{enumerate}

To construct such an extension we will consider the Galois representation arising from the $p$-adic Eichler-Shimura cohomology group of level $N$. Specifically, by applying the method of Kurihara \cite{Kuri} and Harder-Pink \cite{HP} to this representation, we are able to construct a pro-$p$ abelian extension $L/F_\infty$. Without assuming (H1) or (H2) it is possible that this extension is ramified. However, the method of Kurihara and Harder-Pink also supplies us with an embedding of $\mathrm{Gal}(L/F_\infty)$ into the reduction modulo Eisenstein ideal of a particular lattice of the quotient field of Hida's universal ordinary cuspidal Hecke algebra. Through this embedding we are able understand the structure of $\mathrm{Gal}(L/F_\infty)$ as an Iwasawa module. In particular, we can show that $\Delta$ acts on $\mathrm{Gal}(L/F_\infty)$ via $\xi_1$. We then use this structure to determine not only which primes can ramify in the extension $L/F_\infty$, but also how this ramification manifests itself in terms of the characteristic ideal of $\mathrm{Gal}(L/F_\infty)$.

\begin{customlem}{\ref{jul2-2014-307pm}}[\cite{OhtaCM}, Lemma A.2.1] Let $\ell\neq p$ be a prime and $K_\ell/F_\infty$ the maximal subextension of $L/F_\infty$ in which the primes above $\ell$ are unramified. The Galois group $\mathrm{Gal}(L/K_\ell)$ is a cyclic ${\Lambda_\xi}$-module annihilated by $b_\ell(X):=(1+X)^{s(\ell)}  - \xi_1^{-1}(\ell)\ell.$

\end{customlem}

\begin{customlem}{\ref{feb24-2015-257pm}} Let $\ell$ be a prime. If $\ell \nmid N$ or $\xi_2(\ell)$ is not a $p$-power root of unity, then $\ell$ is unramified in $L/F_\infty$.
\end{customlem}

Let $L^\mathrm{un}/F_\infty$ be the maximal unramified subextension of $L/F_\infty$. Using Lemmas \ref{jul2-2014-307pm} and \ref{feb24-2015-257pm} in combination with the theory of Fitting ideals, we will show $\mathrm{Char}_{{\Lambda_\xi}}(\mathrm{Gal}(L^\mathrm{un}/F_\infty)) = (F(X,\xi_2^{-1}))$. With the main conjecture in hand, we conclude by determining the characteristic ideal of $\mathrm{Gal}(L/F_\infty)$.
 
\begin{customlem}{\ref{jun29-2015-313pm}} Let $\tilde{A}$ denote the image of $A_{\theta,\psi}$ under the involution induced by $X\mapsto u^{-1}(1+X)^{-1}-1$. Set $\tilde{A}_0=\tilde{A}/X$ if the pair $(\theta_0,\psi_0)$ is exceptional, with $\tilde{A}_0=\tilde{A}$ otherwise. Then $\mathrm{Char}_{{\Lambda_\xi}}(\mathrm{Gal}(L/F_\infty)) = (\tilde{A}_0)$.
\end{customlem}

 \subsection{Outline} In Section \ref{secHida} we briefly recall notation and results from the theory of classical modular forms and their Hecke algebras that will be needed in subsequent sections. We then describe the construction of Hida's universal ordinary Hecke algebra. 
 
 In Section \ref{secORD} we recall the definition of ordinary $\Lambda$-adic forms following Ohta \cite{OhtaEC1}. After recording several well known results on the structure of the space of ordinary $\Lambda$-adic modular forms, we introduce $\Lambda$-adic Eisenstein series and prove several results pertaining to these forms. 
 
 In Section \ref{secRES} we introduce the ordinary $\Lambda$-adic cuspidal group and the $\Lambda$-adic residue map. We then compute the image of $\Lambda$-adic Eisenstein series under this map. 
 
 In Section \ref{secDUAL} we prove our refinement of Hida's duality theorem, and use this refinement to give a simple proof of the isomorphism $\mathfrak{h}_\Lambda/I_{\theta,\psi}\cong\Lambda/(A_{\theta,\psi})$.  
 
 Finally, in Section \ref{secIWA} we will use the results of the previous section to extend Ohta's proof of the Iwasawa main conjecture over $\mathbb{Q}$.

\subsection{Notation and conventions} We fix embeddings $\overline{\mathbb{Q}}\rightarrow \overline{\mathbb{Q}}_p$ and $\overline{\mathbb{Q}}\rightarrow \mathbb{C}$. Through these embeddings we may consider a Dirichlet character as taking values in $\mathbb{C}$ or $\overline{\mathbb{Q}}_p$. For a field $F$, we set $G_{F} = \mathrm{Gal}(\overline{F}/F)$.

For a character $\chi$ and any positive integer $n$, let $\chi_{(n)}$ denote the (possibly imprimitive) character defined modulo $\mathrm{lcm}(M_\chi,n)$ that is induced from the character $\chi$.

For all integers $r\geq 1$ we set $U_r = 1+p^r\mathbb{Z}_p$. Note that $u=1+p$ is a topological generator of $U_1$.

Finally, for a positive integer $M$ we set $$\Gamma_1(M)=\left\{\begin{pmatrix}a & b\\ c & d\end{pmatrix}\in\mathrm{SL}_2(\mathbb{Z}): a,d\equiv1~(\mathrm{mod}\,M),~c\equiv0~(\mathrm{mod}\,M) \right\}.$$ We let $N_r = Np^r$ and $\Gamma_r = \Gamma_1(Np^r)\subset \mathrm{SL}_2(\mathbb{Z})$.

\subsection{Acknowledgements} The author would like to thank Romyar Sharifi for suggesting this problem, as well as his guidance, insight, and encouragement.

\section{Classical modular forms and Hida's universal ordinary Hecke algebra}\label{secHida}

Let $k$ be a non-negative integer. For a positive integer $r$, we denote the space of holomorphic modular forms (resp., cusp forms) of weight $k$ with respect to $\Gamma_r$ by $M_{k,r}$ (resp., $S_{k,r}$). The weight $k$ action of $\alpha\in\mathrm{GL}_2^+(\mathbb{R})$ on $M_{k,r}$ is defined by $$(f|_k\alpha)(z) ~=~ \det(\alpha)^{k/2}(cz+d)^{-k}\, f(\alpha(z))~~~~~\text{for~}\alpha = \begin{pmatrix} a & b \\ c & d \end{pmatrix}.$$ 

\noindent To make the notation more compact, we will often omit the weight from the notation for this action. We identify each $f\in M_{k,r}$ with its unique $q$-expansion and denote the $n^\mathrm{th}$ coeffeicent of this expansion by $a_n(f)$. We set \begin{align*}M_{k,r,\mathbb{Z}} &= M_{k,r} \cap \mathbb{Z}\llbracket q\rrbracket\\
M_{k,r,\mathcal{O}} &= M_{k,r,\mathbb{Z}} \otimes_\mathbb{Z} \mathcal{O}\end{align*} with $S_{k,r,\mathbb{Z}}$ and $S_{k,r,\mathcal{O}}$ defined analogously.

We now recall the definition of Hecke operators in terms of double cosets. For any $\alpha \in \mathrm{GL}_2^+(\mathbb{Q})$, the double coset $\Gamma_r \alpha\Gamma_r = \coprod_i\Gamma_r \alpha_i$ acts on $M_{k,r}$ as follows: $$f|[\Gamma_r \alpha\Gamma_r] ~:=~ \sum_i f|\alpha_i.$$

\noindent For all integers $n\geq 1$, we denote the operator associated to the double coset $$\Gamma_r\begin{pmatrix}1 & 0 \\ 0 & n\end{pmatrix}\Gamma_r$$ by $T_n$. The operator $T_p$ will be of special significance, and we note that $$\Gamma_r \begin{pmatrix}1 & 0 \\ 0 & p\end{pmatrix}\Gamma_r ~=~ \coprod_{i=0}^{p-1}\Gamma_r\begin{pmatrix}1 & i \\ 0 & p\end{pmatrix}.$$

For $d\in (\mathbb{Z}/N_r\mathbb{Z})^\times$ we define the diamond operator $\langle d\rangle$ (resp., $T_{d,d}$) to be the operator associated to the double coset $\Gamma_r \alpha_d\Gamma_r$ (resp., $\Gamma_r d\alpha_d\Gamma_r$), where $\alpha_d\in \mathrm{SL}_2(\mathbb{Z})$ satisfies $$\alpha_d ~\equiv~\begin{pmatrix}* & * \\ 0 & d\end{pmatrix}~(\mathrm{mod}\,N_r).$$ We extend the definition of these operators to all positive integers $d$ by defining $\langle d\rangle=0=T_{d,d}$ whenever $\gcd(d,N_1)>1$. Having done so, we can describe the action of $T_n$ on $M_{k,r}$ in terms of $q$-expansions: For all integers $m\geq 0$ and $n\geq 1$, $$a_m(f|T_n) ~=~ \sum_{d\mid \gcd(m,n)} a_{mn/d^2}(f|T_{d,d}).$$

 In addition to the operators $T_n$, $T_{d,d}$, and $\langle d\rangle$, we will also consider their adjoints  $T_n^*$, $T_{d,d}^*$, and $\langle d\rangle^*$, which are associated to the double cosets $$\Gamma_r\begin{pmatrix}n & 0 \\ 0 & 1\end{pmatrix}\Gamma_r,$$ $\Gamma_r d\alpha_d^{-1}\Gamma_r$ and $\Gamma_r \alpha_d^{-1}\Gamma_r$, respectively. For future reference, we note that $H^*  = w_{N_r}^{-1} Hw_{N_r}$ for $H=T_n,T_{d,d},$ or $\langle d\rangle$, where $$w_M := \begin{pmatrix}0 & - 1\\ M & 0\end{pmatrix}$$ for all positive integers $M$.

We define $\mathfrak{H}_{k,r}$ (resp., $\mathfrak{h}_{k,r}$) to be the $\mathbb{Z}$-subalgebra of $\mathrm{End}_\mathbb{Z}(M_{k,r})$ (resp., $\mathrm{End}_\mathbb{Z}(S_{k,r})$) generated by the operators $T_n$ and $T_{d,d}$ for all integers $n\geq 1$ and $d\in (\mathbb{Z}/N_r\mathbb{Z})^\times$. Set \begin{align*}\mathfrak{H}_{k,r,\mathcal{O}} &= \mathfrak{H}_{k,r} \otimes_\mathbb{Z} \mathcal{O} \\ \mathfrak{h}_{k,r,\mathcal{O}} &= \mathfrak{h}_{k,r} \otimes_\mathbb{Z} \mathcal{O}.\end{align*} 

\noindent It is well known that $M_{k,r,\mathbb{Z}}$ and $S_{k,r,\mathbb{Z}}$ are stable under the action of $T_n$ and $T_{d,d}$ \cite[\S 1]{Hida86a}. Consequently, $M_{k,r,\mathbb{Z}}$ and $S_{k,r,\mathbb{Z}}$ are modules over $\mathfrak{H}_{k,r,\mathcal{O}}$ and $\mathfrak{h}_{k,r,\mathcal{O}}$, respectively.

We define $\mathfrak{H}_{k,r}^*$ and $\mathfrak{h}_{k,r}^*$ (resp., $\mathfrak{H}_{k,r,\mathcal{O}}^*$ and $\mathfrak{h}_{k,r,\mathcal{O}}^*$) analogously with respect to the adjoint operators $T_n^*$ and $T_{d,d}^*$.

\subsection{Hida's universal ordinary Hecke algebra}

Let $k\geq 2$ and $r\geq1$. Rather than consider the whole space $M_{k,r,\mathcal{O}}$, we will primarily restrict our considerations to the maximal subspace on which the action of the Hecke operator $T_p$ is invertible. We project to this subspace using Hida's idempotent associated to the operator $T_p$, which we denote by $e$. We will also consider $e^*$ which is defined analogously with respect to the operator $T_p^*$.

 The natural injections \begin{align*}
eM_{k,r,\mathcal{O}} &\hookrightarrow eM_{k,r+1,\mathcal{O}}\\ eS_{k,r,\mathcal{O}} &\hookrightarrow eS_{k,r+1,\mathcal{O}}\end{align*}

\noindent commute with the Hecke action. Therefore, if we restrict the operators of $e\mathfrak{H}_{k,r+1,\mathcal{O}}$ (resp., $e\mathfrak{h}_{k,r+1,\mathcal{O}}$) to the image of $eM_{k,r,\mathcal{O}}$ (resp., $eS_{k,r,\mathcal{O}}$) we obtain surjective $\mathcal{O}$-algbera homomorphisms\begin{align}\label{mar11-2015-222pm}
e\mathfrak{H}_{k,r+1,\mathcal{O}} &\twoheadrightarrow e\mathfrak{H}_{k,r,\mathcal{O}}\\
e\mathfrak{h}_{k,r+1,\mathcal{O}} &\twoheadrightarrow e\mathfrak{h}_{k,r,\mathcal{O}}.
\end{align} 

\begin{definition}[\cite{Hida86b}, (1.2)] The universal ordinary  Hecke algebra $($resp., universal ordinary cuspidal Hecke algebra$)$ of level $N$ over $\mathcal{O}$ is defined by $$\mathfrak{H}_\Lambda = \varprojlim_r e\mathfrak{H}_{k,r,\mathcal{O}}\hspace{.2in}
(\mathrm{resp.,~}\mathfrak{h}_\Lambda= \varprojlim_r e\mathfrak{h}_{k,r,\mathcal{O}}),
$$ where the projective limit is taken with respect to the above restriction maps.
\end{definition}

\noindent Hida has shown that the above projective limits are isomorphic for all $k\geq 2$ \cite[Theorem 1.1]{Hida86b}, which is the reason we omit reference to the weight in the notation. We denote the operators corresponding to the projective limits of $T_n$, $T_{d,d}$, and $\langle d\rangle$ by the same symbols.

Let $$\mathbb{Z}_{p,N} ~=~ \varprojlim_r \mathbb{Z}/Np^r\mathbb{Z}  ~\cong~ (\mathbb{Z}/N\mathbb{Z}) \times  \mathbb{Z}_p.$$

\noindent We identify $\mathcal{O}[\mathbb{Z}_{p,N}^\times]\cong \mathcal{O}[(\mathbb{Z}/Np\mathbb{Z})^\times]\llbracket U_1\rrbracket$ with $\mathcal{O}[(\mathbb{Z}/Np\mathbb{Z})^\times]\llbracket X\rrbracket$ through the isomorphism
\begin{align}\label{mar3-2015-1045pm}\iota&: \mathcal{O}[(\mathbb{Z}/Np\mathbb{Z})^\times]\llbracket U_1\rrbracket\rightarrow \mathcal{O}[(\mathbb{Z}/Np\mathbb{Z})^\times]\llbracket X\rrbracket\end{align} induced by $u\mapsto 1+X$. The Hecke algebras $\mathfrak{H}_\Lambda$ and $\mathfrak{h}_\Lambda$ have a natural $\mathcal{O}[\mathbb{Z}_{p,N}^\times]$-algebra structure, in which any integer $d\in(\mathbb{Z}/Np\mathbb{Z})^\times$ acts on $\mathfrak{H}_\Lambda$ (resp., $\mathfrak{h}_\Lambda$) as $T_{d,d}$.

\begin{prop}[\cite{OhtaEC1}, Theorem 1.5.7]\label{mar19-2015-143pm} $\mathfrak{H}_\Lambda$ and $\mathfrak{h}_\Lambda$ are free and finitely generated $\Lambda$-modules.

\end{prop}

We have the following commutative diagram $$\begin{CD}e\mathfrak{H}_{k,r+1,\mathcal{O}} @>\sim>> e^*\mathfrak{H}^*_{k,r+1,\mathcal{O}}\\
@V\mathrm{res}VV @VV\mathrm{res}V\\
e\mathfrak{H}_{k,r,\mathcal{O}} @>\sim>> e^*\mathfrak{H}^*_{k,r,\mathcal{O}} \end{CD}$$ where the horizontal maps are induced by $T_n\mapsto T_n^*$ and the vertical maps are restriction (\ref{mar11-2015-222pm}).  From these isomorphisms we construct the adjoint universal ordinary Hecke algebra $\mathfrak{H}^*_\Lambda$. We construct $\mathfrak{h}^*_\Lambda$ analogously.

\section{Ordinary $\Lambda$-adic modular forms}\label{secORD}

In this section we recall the definition of ordinary $\Lambda$-adic modular forms following Ohta \cite{OhtaEC1}. We then introduce $\Lambda$-adic Eisenstein series and record several results pertaining to these forms.

\subsection{Ordinary $\Lambda$-adic modular forms} Denote the group of continuous $\overline{\mathbb{Q}}^\times$-valued characters on $U_1/U_r$ by $\widehat{U_1/U_r}$, and define $$\widehat{U}_{1}~=~\bigcup_{r\geq 1} \widehat{U_1/U_r}.$$

\noindent We will always assume that the characters $\epsilon\in \widehat{U}_1$ are primitive. For $\epsilon\in \widehat{U}_{1}$ we define
\begin{align*}
eM_{k,r,\mathcal{O},\epsilon} ~=~ \{f\in eM_{k,r,\mathcal{O}[\epsilon]}: f|\sigma_\alpha = \epsilon(\alpha)f~\mathrm{for~all~}\alpha\in U_1  \},
\end{align*}

\noindent where $\sigma_\alpha \in \Gamma_1$ is a matrix satisfying
\begin{align}\label{feb14-2014-308pm}
\sigma_\alpha ~\equiv~ \begin{pmatrix}\alpha^{-1} & * \\ 0 & \alpha\end{pmatrix}~(\mathrm{mod}\,p^r).
\end{align}

\noindent We define $S_{k,r,\mathcal{O},\epsilon}$ analogously.

\begin{definition} An ordinary $\Lambda$-adic modular form $($resp., cusp form$)$ $F$ of level $N$ is a formal $q$-expansion $$F ~=~ \sum_{n=0}^\infty a_n(F)(X) q^n ~\in~\Lambda\llbracket q\rrbracket$$ such that $$v_{k,\epsilon}(F) ~:=~ \sum_{n=0}^\infty a_n(F)(\epsilon(u)u^{k-2}-1) q^n$$ 

\noindent is an element of $eM_{k,r,\mathcal{O},\epsilon}$ $($resp., $eS_{k,r,\mathcal{O},\epsilon})$ for all  $k\geq 2$ and $\epsilon\in \widehat{U}_{1}$. Here the power of $p$ appearing in the level $N_r$ is determined by $\ker(\epsilon)=U_r$. We denote the space of ordinary $\Lambda$-adic modular forms $($resp., cusp forms$)$ of level $N$ by $M_\Lambda$ $($resp., $S_\Lambda)$.\end{definition}

The space of ordinary $\Lambda$-adic modular forms has a very nice structure which we now recall.

\begin{prop}[\cite{HidaBB}, \S 7.3 Theorem 1]\label{dec10-209pm} The $\Lambda$-modules $M_\Lambda$ and $S_\Lambda$ are free and finitely generated.

\end{prop}

\begin{prop}[\cite{OhtaEC1} Proposition 2.5.1, \cite{OhtaES} Proposition 2.6.4]\label{mar3-2015-324pm} For each $k\geq 2$ and $\epsilon\in \widehat{U}_1$, let $P_{k,\epsilon}:= X - \epsilon(u)u^{k-2}+1$. Then
\begin{align*}
M_{\Lambda} / P_{k,\epsilon}M_{\Lambda}~\cong~ eM_{k,r,\mathcal{O},\epsilon}\\
S_{\Lambda} / P_{k,\epsilon}S_{\Lambda}~\cong~ eS_{k,r,\mathcal{O},\epsilon}.
\end{align*} 

\end{prop}

\begin{cor}\label{mar3-2015-332pm} We have \begin{align*}M_{\Lambda} &~\cong~ M_{ \mathbb{Z}_p\llbracket X\rrbracket}\otimes_{ \mathbb{Z}_p\llbracket X\rrbracket} \Lambda\\
S_{\Lambda} &~\cong~ S_{ \mathbb{Z}_p\llbracket X\rrbracket}\otimes_{ \mathbb{Z}_p\llbracket X\rrbracket} \Lambda
\end{align*} \end{cor}

In \cite[\S 2.3]{OhtaES} and \cite[\S 2.2]{OhtaEC1}, Ohta shows that the space of ordinary $\Lambda$-adic modular forms is isomorphic to a projective system of classical modular forms. The latter has a natural $\mathfrak{H}^*_\Lambda$-module structure, and through this isomorphism $M_\Lambda$ is endowed with an $\mathfrak{H}_\Lambda$-module structure. In particular, for all $F\in M_\Lambda$ we have
\begin{align*}
v_{k,\epsilon}(F|T_n) &~=~ v_{k,\epsilon}(F)|T_n\\
v_{k,\epsilon}(F|T_{d,d}) &~=~ v_{k,\epsilon}(F)|T_{d,d}\end{align*}

\noindent for all $k\geq 2$ and $\epsilon\in \widehat{U}_{1}$.

\subsection{\texorpdfstring{$\Lambda$}{Lambda}-adic Eisenstein series}

In this subsection we assume that $\mathcal{O}$ contains the values of $\theta$ and $\psi$.	Let $[\cdot]:\mathbb{Z}_p^\times \rightarrow U_1$ be the projection defined by $[a] = a\omega(a)^{-1}$, and let $s:\mathbb{Z}_p^\times\rightarrow \mathbb{Z}_p$ be the group homomorphism defined by $[a] = u^{s(a)}$. For a Dirichlet character $\varphi$ with conductor not divisible by $p^2$, set $G(X,\varphi\omega^2) = F(u^{-1}(1+X)^{-1}-1,\varphi\omega^2)$. Note that $$G(\epsilon(u)u^{k-2}-1,\varphi\omega^2) ~=~ L_p(1-k,\varphi\omega^2\epsilon) ~=~ L(1-k,(\varphi\omega^{2-k}\epsilon)_{(p)}),$$ for all $k\geq 2$ and $\epsilon\in \widehat{U}_1$, where $L(s,\chi)$ is the Dirichlet $L$-function associated to the character $\chi$ \cite[(2.3.6)]{OhtaEC1}.

For all integers $t\geq 1$ we define the following formal series in $\Lambda\llbracket q\rrbracket$:
\begin{align*}
\mathcal{E}_{\theta,\psi;t}= \delta_{\theta,\psi}(X)\left(\displaystyle{\frac{\psi(0) G(X,\theta\omega^2)}{2} + \sum_{n=1}^\infty \left(\sum_{\substack{0<d\mid n \\ p\nmid d}}\theta(d)\psi\!\left(\frac{n}{d}\right)(1+X)^{s(d)}d\right)q^{tn}}\right).
\end{align*}

\noindent We set $\mathcal{E}_{\theta,\psi}=\mathcal{E}_{\theta,\psi;1}$.

\begin{theorem}[\cite{OhtaCM} \S 1.4, \cite{OhtaEC1} \S2.4]\label{dec31-2013-205pm} The power series $\mathcal{E}_{\theta_0,\psi_0;t}$ is an element of $M_\Lambda$ if the following conditions are satisfied:

\begin{enumerate}[topsep=.1em,itemsep=-.2em]
\item[$\mathrm{(1)}$]~~$p\nmid t$
\item[$\mathrm{(2)}$]~~$f_{\theta} f_{\psi} t\mid Np$

\item[$\mathrm{(3)}$]~~$(f_{\psi},p)=1$

\item[$\mathrm{(4)}$]~~$(\theta_0\psi_0)(-1)=1.$
\end{enumerate}

\noindent For all $k\geq 2$ and $\epsilon\in\widehat{U}_1$, we have

\begin{center}
$\displaystyle{ v_{k,\epsilon}(\mathcal{E}_{\theta_0,\psi_0;t}) ~=~\delta_{\theta,\psi}(\epsilon(u)u^{k-2}-1)   E_k((\theta_0\epsilon\omega^{2-k})_{(p)},\psi_0;t),   }$
\end{center}

\noindent where $E_k((\theta_0\epsilon\omega^{2-k})_{(p)},\psi_0;t)$ is the classical $p$-stabilized Eisenstein series of weight $k$ and level $f_\theta f_\psi p^rt/\gcd(f_\theta,p)$ having Nebentypus $\theta_0\psi_0\epsilon\omega^{2-k}$. 

Furthermore, $M_\Lambda \otimes_\Lambda Q(\Lambda)$ is spanned over $Q(\Lambda)$ by $S_\Lambda$ and the set of Eisenstein series $\mathcal{E}_{\theta_0,\psi_0;t}$ satisfying the above conditions.

\end{theorem}

\begin{prop}\label{dec18-2013-1155am} Let $D_\theta$ and $D_\psi$ be the largest square-free factors of $M_\theta$ and $M_\psi$, respectively, such that $\gcd(D_\theta,f_\theta p) = 1= \gcd(D_\psi,f_\psi ).$ For all integers $t\geq 1$ we have $$ \mathcal{E}_{\theta,\psi;t}~=~ \sum_{\substack{\alpha \mid D_\theta \\ \beta\mid D_\psi}}\alpha\mu(\alpha)\mu(\beta)\theta_0(\alpha)\psi_0(\beta)(1+X)^{s(\alpha)}\mathcal{E}_{\theta_0,\psi_0;\alpha\beta t}$$

\noindent where $\mu$ is the M\"{o}bius function.

\end{prop} 

\begin{proof} Suppose we have the following factorizations of $D_\theta$ and $D_\psi$,
\begin{align*}
D_{\theta} &~=~ p_1\cdots p_m \\
D_{\psi} &~=~ p_{1}^\prime\cdots p_{m^\prime}^\prime\,, 
\end{align*}

\noindent keeping in mind that the sets $\{p_1,\dots,p_m\}$ and $\{p_1^\prime,\dots,p_{m^\prime}^\prime\}$ may not be disjoint. To simplify the notation a bit, for $1\leq i \leq m$ and $1\leq j \leq m^\prime$ define
\begin{align*}
\theta^{(i)} &~=~ \theta_{(p_1\cdots p_i)}\\
\psi^{(j)} &~=~ \psi_{(p_1^\prime\cdots p_j^\prime)},\end{align*}

\noindent  with $\theta^{(0)}=\theta_0$ and $\psi^{(0)}=\psi_0$. We begin by considering the non-constant terms of $\mathcal{E}_{\theta,\psi;t}$. For all $n\geq 1$ and $1\leq i \leq m$

\begin{center}
$\displaystyle{  a_{nt}(\mathcal{E}_{
\theta^{(i-1)},\psi;t}) - a_{nt}(\mathcal{E}_{\theta^{(i)},\psi;t})  ~=~  p_i\theta_0(p_i)(1+X)^{s(p_i)} a_{nt}(\mathcal{E}_{\theta^{(i-1)},\psi;p_it})   }$,
\end{center}

\noindent which gives us the recursive identity

\begin{center}
$\displaystyle{  a_{nt}(\mathcal{E}_{\theta^{(i)},\psi;t})  ~=~ a_{nt}(\mathcal{E}_{\theta^{(i-1)},\psi;t}) + p_i\mu(p_i) \theta_0(p_i)(1+X)^{s(p_i)} a_{nt}(\mathcal{E}_{\theta^{(i-1)},\psi;p_i t})  }$.
\end{center}

\noindent From this identity, we obtain
\begin{align*}
a_{nt}(\mathcal{E}_{\theta,\psi;t}) ~=~  \sum_{\alpha\mid D_\theta}\alpha\mu(\alpha)\theta_0(\alpha)(1+X)^{s(\alpha)}a_{nt}(\mathcal{E}_{\theta_0,\psi;\alpha t}).
\end{align*}

\noindent Next we note that for $1\leq j\leq m^\prime$ we have

\begin{center}
$\displaystyle{  a_{n t}(\mathcal{E}_{
\theta_0,\psi^{(j-1)};\alpha t}) - a_{n t}(\mathcal{E}_{\theta_0,\psi^{(j)};\alpha t}) ~=~ \psi_0(p_j^\prime) a_{n t}(\mathcal{E}_{\theta_0,\psi^{(j-1)};\alpha tp_j^\prime})  }$.
\end{center}

\noindent Applying the same recursive argument as above we obtain the desired result for the non-constant coefficients. 

Finally, by considering the Euler factor expansion of the Kubota-Leopoldt $p$-adic $L$-function, we have
$$G(X,\theta\omega^2) ~=~ \left(\sum_{\alpha\mid D_\theta}\alpha \mu(\alpha)\theta_0(\alpha)(1+X)^{s(\alpha)}\right)\cdot G(X,(\theta\omega^2)_0),$$ and the result follows by noting that $$\sum_{\beta\mid D_\psi} \mu(\beta) ~=~ \left\{\begin{array}{rcl}1 & & D_\psi=1\\ 0 & & D_\psi>1\end{array}\right..$$\end{proof}

\noindent The following corollary is an immediate consequence of Proposition \ref{dec18-2013-1155am} and Theorem \ref{dec31-2013-205pm}.

\begin{cor}\label{jan3-2014-157pm} The power series $\mathcal{E}_{\theta,\psi;t}$ is an element of $M_\Lambda$ if the following conditions are satisfied:\begin{enumerate}[itemsep=0em]
\item[$\mathrm{(1)}$]~~$p\nmid t$
\item[$\mathrm{(2)}$]~~$M_{\theta} M_{\psi} t\mid Np$

\item[$\mathrm{(3)}$]~~$(M_\psi,p)=1$

\item[$\mathrm{(4)}$]~~$(\theta_0\psi_0)(-1)=1.$
\end{enumerate}

\end{cor}

\noindent It is well known that $\Lambda$-adic Eisenstein series are Hecke eigenforms. We recall the following results due to Ohta regarding their eigenvalues.

\begin{prop}[\cite{OhtaCM}, Lemma 1.4.8]\label{sep8-2014-1136am} Suppose $\mathcal{E}_{\theta,\psi;t}\in M_\Lambda$. Then \begin{enumerate}[itemsep=0em]
\item[$(i)$]$\mathcal{E}_{\theta,\psi;t}|T_{d,d} ~=~ (\theta\psi)(d)(1+X)^{s(d)} \cdot\mathcal{E}_{\theta,\psi;t}\hfill (\text{integers $d>0$ prime to }Np)$

\item[$(ii)$]$\mathcal{E}_{\theta,\psi;t}|\langle d\rangle ~=~ (\theta\psi)(d) \cdot\mathcal{E}_{\theta,\psi;t}\hfill (\text{integers $d>0$ prime to }Np) $

\item[$(iii)$]$\mathcal{E}_{\theta,\psi;t}| T_\ell ~=~ (\theta(\ell) \ell (1+X)^{s(\ell)} + \psi(\ell))\cdot\mathcal{E}_{\theta,\psi;t}\hfill (\text{primes }\ell\nmid Np)$

\item[$(iv)$]$\mathcal{E}_{\theta,\psi;t}| T_p ~=~ \psi(p)\cdot\mathcal{E}_{\theta,\psi;t}$
\end{enumerate}\medskip

\noindent If $M_\theta M_\psi = N$ or $Np$ $($consequently $t=1)$, identity $(iii)$ holds for all primes $\ell\neq p$.

\end{prop}

\begin{lemma}[\cite{OhtaCM}, Lemma 1.4.9]\label{jun22-2015-1240pm} Suppose $\mathcal{E}_{\theta_i,\psi_i;t_i}\in M_\Lambda$ for $i=1,2$. The $T_\ell$-eigenvalues of $\mathcal{E}_{\theta_1,\psi_1;t_1}$ and $\mathcal{E}_{\theta_2,\psi_2;t_2}$ are congruent modulo $\mathfrak{m} = (\pi,X)$ for all primes $\ell\nmid Np$ if and only if at least one  of the following conditions is satisfied: \begin{enumerate}[itemsep=0em]
\item[$(1)$] $(\theta_1)_0\equiv (\theta_2)_0~(\mathrm{mod}\,\pi) \text{ and~~}(\psi_1)_0\equiv (\psi_2)_0~(\mathrm{mod}\,\pi),$

\item[$(2)$] $(\theta_1)_0\equiv (\psi_2\omega^{-1})_0~(\mathrm{mod}\,\pi) \text{ and~~}(\psi_1)_0\equiv (\theta_2\omega)_0~(\mathrm{mod}\,\pi).$\end{enumerate}

\end{lemma}

\section{The $\Lambda$-adic residue map}\label{secRES}

The primary goal of this section is to determine the image of Eisenstein series under Ohta's $\Lambda$-adic residue map. As mentioned in the introduction, this image was determined by Ohta for Eisenstein series associated to pairs of primitive, non-exceptional characters, and we would like to generalize this result to pairs of arbitrary characters. The image of Ohta's $\Lambda$-adic residue map lies in the group ring generated over $\Lambda$ by the ordinary $\Lambda$-adic cuspidal group. We begin by describing this group.

\subsection{The ordinary $\Lambda$-adic cuspidal group}

\subsubsection{Cuspidal groups}\label{feb10-2015-1243pm}

Let $M$ be a positive integer. We denote the complete modular curve associated with the group $\Gamma_1(M)$ by $X_1(M)$. Let $C_M$ denote the cusps of $X_1(M)$, which we will identify with $\Gamma_1(M)\setminus \mathbb{P}^1(\mathbb{Q})$. The map $\mathrm{SL}_2(\mathbb{Z}) \rightarrow \mathbb{P}^1(\mathbb{Q})$  defined by

\begin{center}
$\displaystyle{   \begin{pmatrix}a& b \\ c & d\end{pmatrix} \mapsto \frac{a}{c}  }$
\end{center}

\noindent induces a bijection $\Gamma_1(M)\setminus \mathrm{SL}_2(\mathbb{Z})\,/\,\mathfrak{I}_\infty  \rightarrow\Gamma_1(M)\setminus \mathbb{P}^1(\mathbb{Q})$, where $\mathfrak{I}_\infty\subset \mathrm{SL}_2(\mathbb{Z})$ is the isotropy subgroup of the cusp at $\infty$. Let
\begin{align*}
\displaystyle{ A_M ~:=~ \left\{\begin{bmatrix} x \\ y\end{bmatrix} \in (\mathbb{Z}/M\mathbb{Z})^{2}:  \mathrm{gcd}(x,y)=1\right\}/\sim},
\end{align*}

\noindent where

\begin{center}
$\displaystyle{  \begin{bmatrix} x \\ y\end{bmatrix} ~\sim~ \begin{bmatrix} x^\prime \\ y^\prime\end{bmatrix} }~~\Longleftrightarrow~~ \begin{array}{l}x \equiv x^\prime~(\mathrm{mod}\,\gcd(M,y))\\
y\equiv y^\prime~(\mathrm{mod}\,M)\end{array}$
\end{center}

\noindent and set

\begin{center}
$\begin{bmatrix}a \\ c\end{bmatrix}_{M}~=~ \text{class~of~}\begin{bmatrix}a \\ c\end{bmatrix}$\text{~in~}$A_M$.
\end{center}

\noindent Then we have a natural bijection between $ \Gamma_1(M)\setminus \mathrm{SL}_2(\mathbb{Z})\,/\,\mathfrak{I}_\infty$ and $A_{M}/\{\pm 1\}$ induced by the map

\begin{center}
$\displaystyle{   \begin{pmatrix}a & b \\ c & d\end{pmatrix} ~\mapsto~\begin{bmatrix}a\\c\end{bmatrix}_M~\mathrm{mod}\,\{\pm 1\}   }$.
\end{center}

\noindent We identify $C_M$ with $A_M/\{\pm 1\}$ through this bijection. To simplify notation, let$$\begin{bmatrix}a\\c\end{bmatrix}_M^{\prime}~=~\begin{bmatrix}a\\c\end{bmatrix}_M~\mathrm{mod}\,\{\pm 1\}.$$

For any two coprime integers $M_1$ and $M_2$ satisfying $M=M_1M_2$, there are bijections $$\Gamma_1(M)\setminus \mathrm{SL}_2(\mathbb{Z})\,/\,\mathfrak{I}_\infty ~\xrightarrow{\sim}~ A_{M}/\{\pm 1\} ~\xrightarrow{\sim}~ (A_{M_1}\times A_{M_2})/\{\pm 1\}$$ induced by the maps $$\begin{pmatrix}a & b \\ c & d\end{pmatrix} ~\mapsto~\begin{bmatrix}a\\c\end{bmatrix}_M^\prime ~\mapsto~ \left(\begin{bmatrix}a\\c\end{bmatrix}_{M_1}, \begin{bmatrix}a\\c\end{bmatrix}_{M_2}\right)^{\!\!\prime}.$$ Unfortunately, this decomposition does not hold with respect to $C_M$, that is, in general

\begin{center}
$\displaystyle{  (A_{M_1}\times A_{M_2})/\{\pm1\}~\neq~ (A_{M_1}/\{\pm 1\}) \times (A_{M_2}/\{\pm1\})  }$.
\end{center}

\noindent For this reason we will often work directly with $A_M$ and then reduce modulo $\{\pm 1\}$ to obtain an element of $C_M$. 

Finally, for a ring $R$, let $R[A_{M}]$ denote the free $R$-module generated by $A_{M}$. By the decomposition above, we have

\begin{center}
$\displaystyle{  R[A_M]~\cong~ R[A_{M_1}]\otimes_R R[A_{M_2}]  }$.
\end{center}

\noindent We then define $R[C_M]$ as the quotient of $R[A_M]$ by the $R$-submodule generated by the set $\{\mathfrak{a} - (-1)\mathfrak{a}: \mathfrak{a}\in A_M\}.$

\subsubsection{Hecke operators acting on cuspidal groups}\label{feb14-2014-1243pm}

Let $r\geq 1$. To simplify the notation a bit, set $A_r = A_{N_r}$ and $C_r=C_{N_r}$. In this section we will consider the action of Hecke operators on $\mathcal{O}[A_{r}]$ and $\mathcal{O}[C_{r}]$.

Let $n\geq 1$. Recall that the Hecke operator $T_n$ was defined in terms of the double coset $$\Gamma_r\begin{pmatrix} 1 & 0 \\ 0 & n\end{pmatrix}\Gamma_r ~:=~ \coprod_{i}\Gamma_r \alpha_i.$$ We define the action of $T_n$ on $\mathcal{O}[A_{r}]$ by 
\begin{align}\label{jun8-2015-1258pm}
T_n \begin{bmatrix}a \\c \end{bmatrix}_{N_r}  &~=~   \sum_i  \alpha_i \begin{bmatrix}a\\c \end{bmatrix}_{N_r}.
\end{align}

\noindent Similarly, for any $d\in (\mathbb{Z}/N_r\mathbb{Z})^\times$ we define \begin{align*}
\langle d\rangle \begin{bmatrix}a \\c \end{bmatrix}_{N_r}  &~=~   \begin{bmatrix}d^\prime a\\d c \end{bmatrix}_{N_r},
\end{align*}

\noindent where $d^\prime$ is an integer such that $dd^\prime\equiv 1~(\mathrm{mod}\,N_r)$. From the definition we see that the action of $(\mathbb{Z}/N_r\mathbb{Z})^\times\cong (\mathbb{Z}/N\mathbb{Z})^\times \times (\mathbb{Z}/p^r\mathbb{Z})^\times$ via the diamond operator is compatible with the decomposition $A_{r}=A_{N}\times A_{p^r}$.

We remark that our notation for the operator defined by (\ref{jun8-2015-1258pm}) differs from that of Ohta in \cite{OhtaEC1} and \cite{OhtaCM}, where this operator is denoted by $T_n^*$. The reason for this difference in notation stems from the fact that Ohta identifies the cuspidal group $\mathcal{O}[C_r]$ with its dual group $\mathrm{Hom}(\mathcal{O}[C_r],\mathcal{O})$ via the perfect pairing $$ \mathcal{O}[C_r]\times \mathcal{O}[C_r]\rightarrow \mathcal{O}~:\left(\sum_{\mathfrak{c}\in C_r}a_\mathfrak{c}\mathfrak{c},\sum_{\mathfrak{c}\in C_r}b_\mathfrak{c}\mathfrak{c}\right)\mapsto \sum_{\mathfrak{c}\in C_r}a_\mathfrak{c}b_\mathfrak{c}.$$ One can show that under this identification the action of the adjoint operator $T_n^*$ is given by the double coset defining $T_n$ \cite[Proposition 3.4.12]{OhtaEC1}.

The above operators induce operators on $\mathcal{O}[C_{r}]$ via the projection mapping $\mathcal{O}[A_{r}]\twoheadrightarrow \mathcal{O}[C_{r}]$, which we will denote by the same symbols. Set $A_{r}^\mathrm{ord}=eA_{r}$ and $C_{r}^\mathrm{ord}=eC_{r}$.

\begin{prop}[\cite{OhtaEC1} Prop. 4.3.4, \cite{OhtaCM} (2.2.3)]\label{feb5-2014-114pm} Let $$D_{r}~=~  \left\{\begin{bmatrix}a\\ c\end{bmatrix}_{N_r} \in A_{r}:p\mid c \right\}.$$ Then $\mathcal{O}[A_{r}^\mathrm{ord}]\cong \mathcal{O}[A_{r}]/\mathcal{O}[D_{r}]$.

\end{prop}

Consider the set $$A_{r}^0 ~=~ \left\{\left( \begin{bmatrix}a\\ c\end{bmatrix}_{N}, \begin{bmatrix}0 \\ \omega(c)\end{bmatrix}_{p^r}\right)\in A_{r}: \begin{array}{l}
0 < c < Np,~\gcd(c,p)=1 \\ 
0 \leq a < \gcd(N,c)\end{array} \right\}.$$

\begin{prop}\label{jan31-2014-355pm} $\{\mathfrak{a}\in A_{r}:e\mathfrak{a}\neq 0\}=\{\sigma_\gamma\mathfrak{a}: \gamma \in U_1/U_r,\,\mathfrak{a}\in A_{r}^0 \}.$

\end{prop}

\begin{proof} From the definition of $A_r^0$ it is clear that the elements of $\{\sigma_\gamma\mathfrak{a}: \gamma \in U_1/U_r,\,\mathfrak{a}\in A_{r}^0 \}$ are distinct. Furthermore, by Proposition \ref{feb5-2014-114pm} and the fact that Hida's idempotent $e$ commutes with diamond operators, we know  that $$\{\sigma_\gamma\mathfrak{a}: \gamma \in U_1/U_r,\mathfrak{a}\in A_{r}^0 \}~\subset~ \{\mathfrak{a}\in A_{r}:e\mathfrak{a}\neq 0\}.$$ Let $\mathfrak{a}\in A_{r}$ with $e\mathfrak{a}\neq0$. Then once again by Proposition \ref{feb5-2014-114pm} we know $$\mathfrak{a} ~=~ \left(\begin{bmatrix}a \\ c\end{bmatrix}_{N},\begin{bmatrix}0 \\ c\end{bmatrix}_{p^r}\right)~ =~ \sigma_{[c]}\cdot\left(  \begin{bmatrix}a \\ c\end{bmatrix}_{N},\begin{bmatrix}0 \\ \omega(c)\end{bmatrix}_{p^r}\right),$$ where $\sigma_{[c]}$ is as defined in (\ref{feb14-2014-308pm}).
\end{proof}

\noindent We define $C_{r}^0 = A_{r}^0/\{\pm1\}$.

\subsubsection{The ordinary \texorpdfstring{$\Lambda$}{Lambda}-adic cuspidal group}\label{feb18-2014-628pm}

For all $s\geq r\geq 1$, the map
\begin{align*}
\begin{bmatrix} a\\ c\end{bmatrix}_{N_s}^\prime \mapsto  \begin{bmatrix} a\\ c\end{bmatrix}_{N_r}^\prime
\end{align*}

\noindent induces a surjection $\mathcal{O}[C_s^\mathrm{ord}]=\mathcal{O}[U_1/U_s][C_s^0]\twoheadrightarrow \mathcal{O}[U_1/U_r][C_r^0]=\mathcal{O}[C_r^\mathrm{ord}]$. Furthermore, from Subsection \ref{feb14-2014-1243pm} we see that the Hecke action commutes with these surjections. We define the $\Lambda$-adic cuspidal group by $$C_\Lambda ~=~\varprojlim_{r\geq 1} \mathcal{O}[C_{r}^\mathrm{ord}] ~=~\varprojlim_{r\geq 1} \mathcal{O}[U_1/U_r][C_{r}^0].$$ From its definition we see that $C_\Lambda$ is a $\mathfrak{H}_\Lambda$-module.

\subsection{Residues of \texorpdfstring{$\Lambda$}{Lambda}-adic Eisenstein series}\label{feb3-2015-215pm}

For the remainder of this section we will assume that $\mathcal{O}$ contains the values of $\theta$ and $\psi$. Let $\mathcal{O}_\infty$ denote the ring of integers of a complete subfield of $\mathbb{C}_p$ containing \emph{all roots of unity}. We set $\Lambda_\infty=\mathcal{O}_\infty\llbracket X\rrbracket$.

In \cite{OhtaCM}, Ohta constructs the following exact sequence of $\mathfrak{H}_{\Lambda_\infty}$-modules
\begin{align*}
\begin{CD}
0 @>>>  S_{\Lambda_\infty} 
@>>> M_{\Lambda_\infty}@>\mathrm{Res}_{\Lambda}>> C_{\Lambda_\infty} @>>> 0,
\end{CD}
\end{align*}

\noindent where $\mathrm{Res}_\Lambda$ is the $\Lambda$-adic residue map of level $N$, defined explicitly by
\begin{align*}
\mathrm{Res}_{\Lambda}(F) &= \varprojlim_{r\geq 1} \left(\, \frac{1}{p^{r-1}}\sum_{\mathfrak{c}\in C_{r}}\left(\sum_{\epsilon\in \widehat{U_1/U_r}}\mathrm{Res}_{\mathfrak{c}}\!\left(v_{2,\epsilon}(F)|T_p^{-r}|w_{N_r}^{-1}\right)\right)\cdot e\mathfrak{c}\right),
\end{align*} with $\mathrm{Res}_\mathfrak{c}(f)$ denoting the residue of the differential $\omega_f = f\frac{dq}{q}$ at the cusp $\mathfrak{c}$. Our primary goal for the remainder of this section is to prove the following proposition.

\begin{prop}\label{feb3-2014-705pm} Suppose $\mathcal{E}_{\theta_0,\psi_0;t}\in M_\Lambda$. Then $$\mathrm{Res}_\Lambda(\mathcal{E}_{\theta_0,\psi_0;t}) ~=~ A_{\theta,\psi}\cdot  \mathfrak{e}_{\theta_0,\psi_0;t}$$ for an explicitly determined $\mathfrak{e}_{\theta_0,\psi_0;t}\in C_\Lambda$ with $\mathfrak{e}_{\theta_0,\psi_0;t}\not\in \mathfrak{m}C_\Lambda$. We set $\mathfrak{e}_{\theta_0,\psi_0}=\mathfrak{e}_{\theta_0,\psi_0;1}$.
 
\end{prop}

\noindent Let us determine exactly what computing this residue will entail. To simplify notation set 
$$E_{2}(\epsilon;t)~=~v_{2,\epsilon}(\mathcal{E}_{\theta_0,\psi_0;t})~=~\delta_{\theta,\psi}(\epsilon(u)-1)  \cdot E_2((\theta_0\epsilon)_{(p)},\psi_0;t)$$

\noindent with $E_{2}(\epsilon):=E_{2}(\epsilon;1)$. Using the fact that $E_2(\epsilon;t)$ is a $T_p$-eigenform with eigenvalue $\psi_0(p)$, and $E_{2}(\epsilon;t)|w_{N_r}^{-1} = (1/t)\cdot E_{2}(\epsilon)|w_{N_r/t}^{-1}$, we have
\begin{align*}
\mathrm{Res}_\Lambda(\mathcal{E}_{\theta_0,\psi_0;t})~=~\displaystyle{\varprojlim_{r\geq 1} \left(\frac{\psi_0(p)^{-r}}{tp^{r-1}}\sum_{\substack{\mathfrak{c}\in C_{r}\\\epsilon\in \widehat{U_1/U_r}}}\mathrm{Res}_{\mathfrak{c}}\left(E_{2}(\epsilon)|w_{N_r/t}^{-1}\right)\cdot e\mathfrak{c} \right)}.
\end{align*}

\noindent Because the above sum is over those cusps $\mathfrak{c}\in C_{r}$ satisfying $e\mathfrak{c}\neq 0$,  Proposition \ref{jan31-2014-355pm} tells us that the above can be written as
\begin{align*}
\displaystyle{ \varprojlim_{r\geq 1} \left(\frac{\psi_0(p)^{-r}}{tp^{r-1}}\sum_{\substack{\mathfrak{c}\in C_{r}^{0}\\\epsilon\in \widehat{U_1/U_r}\\\gamma \in U_1/U_r}}\!\!\epsilon^{-1}(\gamma)\cdot\mathrm{Res}_{w_{N_r/t}( \mathfrak{c})}\left(E_{2}(\epsilon)\right)\cdot  \left(\sigma_\gamma\cdot e\mathfrak{c}\right) \right),}
\end{align*}

\noindent where we have used the fact that diamond operators commute with the idempotent $e$, and the identity $f|w_{N_r/t}^{-1}|\sigma_\gamma = f|\sigma_\gamma^{-1}|w_{N_r/t}^{-1}$ holds for all $f\in M_2(\Gamma_1(N_r/t))_{\mathbb{C}_p}$. Our task is then to determine the residue of $E_2(\epsilon)$ at the cusps $w_{N_r/t}(\mathfrak{c})$ for $\mathfrak{c}\in C_{r}^0$. The following definition and proposition give us a simple means of doing so.

\begin{definition} Let $\gamma \in \mathrm{SL}_2(\mathbb{Z})$ correspond to the cusp $\mathfrak{c}\in C_{r}$. The minimal choice of $W>0$ such that $$\begin{pmatrix}1 & W \\ 0 & 1\end{pmatrix}  ~\in~ \gamma^{-1}\Gamma_r\gamma$$ is called the width of the cusp $\mathfrak{c}$.

\end{definition}

\begin{prop}[\cite{OhtaEC1}, Section 4.5]\label{jan28-2015-1038am} Let $\Gamma$ be a congruence subgroup of $\mathrm{SL}_2(\mathbb{Z})$ and let $f\in M_2(\Gamma)$. Then $\mathrm{Res}_\mathfrak{c}(f) = W_\mathfrak{c} \cdot a_0(f|_\mathfrak{c})$, where $W_\mathfrak{c}$ denotes the width of the cusp $\mathfrak{c}$ and $a_0(f|_\mathfrak{c})$ is the constant term of $f$ at $\mathfrak{c}$.

\end{prop}

In Subsection \ref{feb1-2014-357pm} we will compute the constant term of $E_2(\epsilon)$ at the cusp $w_{N_r/t}(\mathfrak{c})$ for all $\mathfrak{c}\in C_{r}^0$. We will then determine the width of these cusps in Subsection \ref{jul3-2013-1106am}. In Subsection \ref{jul3-2013-1108am} we will put all of this together to prove Proposition \ref{feb3-2014-705pm}. Finally,  in Subsection \ref{jul3-2013-1109am} we will consider the image of Eisenstein series associated to imprimitive characters.

\subsubsection{The constant term of Eisenstein series at the cusps}\label{feb1-2014-357pm}

For this subsection we fix an $\epsilon\in \widehat{U}_{1}$ and define $r\geq 1$ to be the integer satisfying $\ker(\epsilon) = U_r$. We begin by recalling the following result due to Ohta.

\begin{prop}[\cite{OhtaCM}, Prop. 2.5.5, Cor. 2.5.7]\label{jul8-2013-1125am} Let $\mathfrak{c}\in C_{r}$ with $$\mathfrak{c} ~=~ \begin{bmatrix}a \\ c\end{bmatrix}_{N_r}^\prime.$$

\noindent If $f_{\theta_0\epsilon} \mid c$, then the constant term of $E_2(\epsilon)$ at the cusp $\mathfrak{c}$ is given by 
\begin{multline*}
\frac{1}{2}\frac{g(\xi\epsilon^{-1})}{g((\theta_0\epsilon)^{-1})}\left(\frac{f_{\theta\epsilon}}{f_{\xi\epsilon^{-1}}} \right)^2\hspace{-5pt}\psi_0\!\left(-\frac{c}{f_{\theta\epsilon}}\right)(\theta_0\epsilon)^{-1}(a)\,\delta_{\theta,\psi}(\epsilon(u)-1)\\
\cdot \left( \prod_{\substack{\ell \mid f_{\theta\epsilon} f_{\psi} \\\ell\nmid f_{\xi\epsilon^{-1}}  }}\left(1-(\xi\epsilon^{-1})(\ell)\ell^{-2}\right) \right) L_p(-1,\xi_2^{-1}\epsilon)
\end{multline*}

\noindent where $g(\chi)$ is the Gauss sum of the character $\chi$. If $f_{\theta_0\epsilon} \nmid c$, then the constant term is $0$.

\end{prop}

Let $N = \tilde{f}_\theta PQ t$, where $\tilde{f}_\theta = f_\theta/\gcd(f_\theta,p)$ and
\begin{align*}
P ~:=~ \prod_{\substack{\ell \mid f_{\psi}\\ \ell\,\mathrm{prime}  }}     \ell^{\mathrm{ord}_\ell(N/\tilde{f}_{\theta} t)}.
\end{align*}

\noindent Let $\mathfrak{c}\in C_{r}^0$  with 
\begin{align*}
\displaystyle{  \mathfrak{c}~=~\begin{bmatrix}  a \\ c\end{bmatrix}_{N_r}^\prime ~=~ \left(\begin{bmatrix}  a \\ c\end{bmatrix}_{N}, \begin{bmatrix}  0 \\ \omega(c)\end{bmatrix}_{p^r}\right)^\prime.}
\end{align*}

\noindent By the definition of $C_{r}^0$, we know that $0\leq a<\gcd(N,c)$ and $0<c<N_1$ with $\gcd(c,ap)=1$. The cusp we are interested in is $w_{N_r/t}(\mathfrak{c})$, which is given by 

\begin{center}$\displaystyle{ w_{N_r/t}(\mathfrak{c}) ~=~\begin{bmatrix}-c/\Delta \\ aN_r/\Delta t \end{bmatrix}_{N_r}^\prime   ~=~\begin{bmatrix}-c/\Delta \\ a\tilde{f}_\theta PQp^r/\Delta \end{bmatrix}_{N_r}^\prime  }$\end{center}

\noindent where $\Delta := \text{gcd}(aN_r/t,c)=\text{gcd}(\tilde{f}_\theta PQ,c)$. 
By Proposition \ref{jul8-2013-1125am}, in order for the constant term of $E_2(\epsilon)$ to be non-zero at the cusp $w_{N_r/t}(\mathfrak{c})$, the following conditions must be satisfied:

\begin{enumerate}[itemsep=0em]

\item[(1)] $\displaystyle{f_{\theta_0\epsilon}~\mid~ a\tilde{f}_\theta PQp^r/\Delta}$

\item[(2)]
$\displaystyle{    \psi_0\left(a\tilde{f}_\theta PQp^r/f_{\theta\epsilon} \Delta\right) ~\neq~ 0.}$

\item[(3)]
$\displaystyle{    (\theta_0\epsilon)^{-1}\left(c/\Delta\right)~\neq~ 0.}$

\end{enumerate}

We want to unravel these conditions in order to get characterizations of $a,~c$ and $\Delta$. Let us begin with $\Delta$. Note that since $f_{\theta\epsilon}=\tilde{f}_\theta p^r$ and $\text{gcd}(c,ap)=1$, condition (1) is equivalent to $\Delta \mid PQ$. Furthermore, by expanding condition (2), $$\displaystyle{\psi_0(a)\,\psi_0\!\left(\frac{PQ}{\Delta}\right)~\neq~ 0}$$ we see that we must also have $P\mid \Delta$. Hence, $\Delta = Pd_Q$ for some $d_Q\mid Q$.

Next we consider $c$. First we write $c = \gcd(N,c)\cdot x$ for some $x$ satisfying $0<x<N_1/\gcd(N,c)$ with $\gcd(x,N_1/\gcd(N,c))=1$. The quotient $\gcd(N,c)/\Delta = \gcd(N,c)/Pd_Q$ must be a factor of $t$, which we denote by $d_t$. Then $c = d_td_QPx$. Moreover, by condition $(3)$ the quotient $c/\Delta = d_tx$ must be prime to $f_\theta p$. Since $\gcd(x,f_\theta p)\mid \gcd(x,N_1/Pd_Qd_t)=1$, we must have $\gcd(d_t,f_{\theta} p)=1$.

Finally, we consider $a$. By definition we know that $0\leq a < P d_t d_Q$. Furthermore, we have $\gcd(a,Pd_td_Q)\mid \gcd(a,Pd_td_Qy)=\gcd(a,c)=1$.

Putting all of this together, we see that if the constant term of $E_2(\epsilon)$ is to be non-zero at $w_{N_r/t}(\mathfrak{c})$, the cusp $\mathfrak{c}\in C_r^0$ can be written as
\begin{align}\label{feb1-2014-239pm}
\displaystyle{ \mathfrak{c} ~=~\left(\begin{bmatrix}y \\ d_td_QP x\end{bmatrix}_{N},  \begin{bmatrix}0 \\ \omega(d_td_QP x) \end{bmatrix}_{p^r}\right)^\prime  }
\end{align}

\noindent where
\begin{enumerate}[itemsep=0em]

\item[(1)] $d_t\mid t$ and $d_Q \mid Q$ with $\gcd(d_t,f_\theta p)=1$,

\item[(2)] $0<x<N_1/d_td_QP$ with $\gcd(x,N_1/d_td_QP)=1$,

\item[(3)] $0<y<d_td_QP$ with $\gcd(y,d_td_QPx)=1$.

\end{enumerate}

\noindent To simplify subsequent notation, we will denote the set of tuples $(d_t,d_Q,x,y)$ satisfying the above conditions for a given $t$ by $\mathscr{S}_t$.

Having characterized the cusps in $C_r^0$ at which the constant term of the Eisenstein series $E_2(\epsilon)$ can be non-trivial, we will now use Proposition \ref{jul8-2013-1125am} to determine the constant term of $E_2(\epsilon)$ at these cusps.

\begin{lemma}\label{feb3-2014-446pm} Suppose $\theta = \chi\omega^i$, with $p\nmid M_\chi$. Let $t\geq 1$ be coprime to $p$ with $f_\theta f_\psi t\nmid Np$. If $\mathfrak{c}\in C_{r}^0$ is the cusp associated to the tuple $(d_t,d_Q,x,y)\in\mathscr{S}_t$ by $(\ref{feb1-2014-239pm})$, then the constant term of $E_2(\epsilon)$ at $w_{N_r/t}(\mathfrak{c})$ is $$C\,\epsilon\left(\frac{f_\theta }{f_{\xi}d_tx}\right)\psi_0\!\left(\frac{yQp^r}{d_Q}\right)\theta_0^{-1}\!(d_tx)\,  \delta_{\theta,\psi}(\epsilon(u) - 1) \left(\prod_{\substack{\ell\,\mid\, f_\theta f_\psi \\ \ell\,\nmid\, f_{\xi}}}(1-(\xi\epsilon^{-1})(\ell)\ell^{-2})\right)L_p(-1,\xi_2^{-1}\epsilon)$$

\noindent where $C$ is a $p$-adic unit in some finite cyclotomic extension of $\mathbb{Q}_p$ depending only on $\theta_0$ and $\psi_0$.

\end{lemma}

\begin{proof} We have
\begin{align*}
\displaystyle{ w_{N_r/t}(\mathfrak{c}) ~=~ \begin{bmatrix}-d_tx\\ y\tilde{f}_\theta Qp^r/d_Q\end{bmatrix}_{N_r}^\prime}.
\end{align*}

\noindent By Proposition \ref{jul8-2013-1125am} the constant coefficient of $E_2(\epsilon)$ at the cusp $w_{N_r/t}(\mathfrak{c})$ is given by
\begin{multline}\label{jun26-2013-352pm}
\displaystyle{ \frac{1}{2}\frac{g(\xi\epsilon^{-1})}{g((\theta_0\epsilon)^{-1})} \left(\frac{f_{\theta\epsilon}}{f_{\xi\epsilon^{-1}}}\right)^2\!\!\!\psi_0\!\left(\frac{-yQ}{d_Q}\right)(\theta_0\epsilon)^{-1}(-d_tx)}\\
\displaystyle{\cdot \left(\prod_{\substack{\ell\mid f_{\theta\epsilon}f_{\psi}\\ \ell\nmid f_{\xi\epsilon^{-1}}}}(1-(\xi\epsilon^{-1})(\ell)\ell^{-2})\right)L_p(-1,\xi_2^{-1}\epsilon)   }.
\end{multline}

Let us consider the first term of (\ref{jun26-2013-352pm}) involving Gauss sums. If $\eta$ and $\varphi$ are two primitive Dirichlet characters with $\gcd(f_\eta,f_\varphi)=1$, we have $g(\eta\varphi) = g(\eta)g(\varphi)\eta(f_{\varphi})\varphi(f_{\eta}).$ Consequently, $$ \frac{g(\xi\epsilon^{-1})}{g((\theta_0\epsilon)^{-1})} ~=~\frac{g((\psi\chi^{-1})_0)}{g(\chi_0^{-1})} \cdot\frac{g((\omega^i\epsilon)^{-1})}{g((\omega^i\epsilon)^{-1})}\cdot \frac{(\psi\chi^{-1})_0(p^r)}{\chi_0^{-1}(p^r)} \cdot\frac{(\omega^i\epsilon)^{-1}\left(f_{\psi\chi^{-1}}\right)}{(\omega^i\epsilon)^{-1}(f_\chi)}$$$$=~ \frac{g((\psi\chi^{-1})_0)}{g(\chi_0^{-1})}\cdot\psi_0(p^r)\cdot \omega^i\left(\frac{f_\chi}{f_{\psi\chi^{-1}}}\right)\cdot \epsilon\left(\frac{f_\chi}{f_{\psi\chi^{-1}}}\right).$$

\noindent The second term of (\ref{jun26-2013-352pm}) can be written as $$\left(\frac{f_{\theta\epsilon}}{f_{\xi\epsilon^{-1}}}\right)^2 ~=~ \left(\frac{f_\chi\cdot f_{\omega^i\epsilon}   }{f_{\psi\chi^{-1}}\cdot f_{\omega^i\epsilon}}\right)^2~=~  \left(\frac{f_\chi }{f_{\psi\chi^{-1}}}\right)^2.$$

\noindent Recalling that $\gcd(f_\psi d_tx,p)=1$, the terms involving $(\theta_0\epsilon)^{-1}$ and $\psi_0$ can be written as, $$\psi_0\!\left(\frac{-yQ}{d_Q}\right)(\theta_0\epsilon)^{-1}(-d_tx) ~=~\psi_0\!\left(\frac{yQ}{d_Q}\right)\theta_0^{-1}(d_tx)\epsilon^{-1}(d_tx),$$

\noindent where we are using the fact that $(\theta_0\psi_0)(-1)=1=\epsilon(-1)$. Putting all of this together, we see that the first half of (\ref{jun26-2013-352pm}) can be written as $$C\cdot \epsilon\left(\frac{f_\theta }{f_{\xi}d_tx}\right)\cdot \theta_0^{-1}(d_tx)\cdot \psi_0\left(\frac{yQp^r}{d_Q}\right)$$

\noindent where $$C ~:=~\frac{1}{2}\cdot\frac{g((\psi\chi^{-1})_0)}{g(\chi_0^{-1})}\cdot \omega^i\!\left(\frac{f_\theta}{f_{\xi}}\right)\cdot\left(\frac{f_\theta}{f_{\xi}}\right)^2$$

\noindent is a $p$-adic unit in some finite cyclotomic extension of $\mathbb{Q}_p$ that depends only on $\theta_0$ and $\psi_0$.

Finally, since $f_{\theta\epsilon} f_{\psi} = f_\theta f_\psi p^r$ and $f_{\xi\epsilon}=f_{\xi}p^{r}$, we have $\{\ell~\mathrm{prime}: \ell \mid f_{\theta\epsilon}f_\psi, \ell\nmid f_{\xi\epsilon} \} = \{\ell~\mathrm{prime}: \ell \mid f_\theta f_\psi, \ell\nmid f_\xi \}$
\end{proof}

\subsubsection{Width of the cusps}\label{jul3-2013-1106am}

In this subsection we would like to determine the width of the cusp $w_{N_r/t}(\mathfrak{c})$, where $\mathfrak{c}$ is associated to the tuple $(d_t,d_Q,x,y)\in\mathscr{S}_t$, i.e.$$\mathfrak{c} ~=~ \begin{bmatrix}y \\ d_td_QPx  \end{bmatrix}_{N_r}^\prime.$$

\noindent Let $\gamma\in \mathrm{SL}_2(\mathbb{Z})$ correspond to $w_{N_r/t}(\mathfrak{c})$, that is,

\begin{center}$\gamma ~=~ \begin{pmatrix}-d_tx & * \\ y\tilde{f}_\theta Qp^r/ d_Q & * \end{pmatrix} $.\end{center}

\noindent Let $W$ be the width of the cusp $w_{N_r/t}(\mathfrak{c})$. Then by definition
\begin{align*}
\begin{pmatrix}1  - d_tx\left(\frac{y\tilde{f}_\theta Qp^r}{d_Q}\right)W &*\\
-\left(\frac{y\tilde{f}_\theta Qp^r}{d_Q}\right)^2W  &1 + d_tx\left(\frac{y\tilde{f}_\theta Qp^r}{d_Q}\right)W   \end{pmatrix} ~=~   \gamma \begin{pmatrix}1 & W \\ 0 & 1\end{pmatrix} \gamma^{-1}  ~\in~ \Gamma_r.
\end{align*}

\noindent Therefore, we must have $$d_tx\left(\frac{y\tilde{f}_\theta Qp^r}{d_Q}\right) W ~\equiv~ 0~(\mathrm{mod}\,N_r)\hspace{.25in}\text{and}\hspace{.25in}\left(\frac{y\tilde{f}_\theta Qp^r}{d_Q}\right)^2 W ~\equiv~ 0~(\mathrm{mod}\,N_r).$$ However, since $\gamma\in \mathrm{SL}_2(\mathbb{Z})$ it must be the case that $$\gcd\left(d_tx,\frac{y\tilde{f}_\theta Qp^r}{d_Q}\right)~=~1,$$ which implies

\begin{center}$\displaystyle{\left(\frac{y\tilde{f}_\theta Qp^r}{d_Q}\right) W  ~\equiv~ 0~(\mathrm{mod}\,N_r)}$.\end{center}

\noindent The smallest value of $W$ satisfying the above congruence is $W= t d_QP/\text{gcd}(y,t)$, which is a $p$-adic unit dependent on the cusp $\mathfrak{c}$.

\subsubsection{Proof of Proposition $\text{\ref{feb3-2014-705pm}}$}\label{jul3-2013-1108am}

Recall that the level $r\geq 1$ component of the projective limit defining $\mathrm{Res}_\Lambda(\mathcal{E}_{\theta_0,\psi_0;t})$ is given by
\begin{align}\label{sep29-2014-950am}
\frac{\psi_0(p)^{-r}}{tp^{r-1}}\sum_{\substack{\mathfrak{c}\in C_{r}^{0}\\\epsilon\in \widehat{U_1/U_r}\\\gamma \in U_1/U_r}}\epsilon^{-1}(\gamma)\,\mathrm{Res}_{w_{N_r/t}( \mathfrak{c})}\!\left(E_{2}(\epsilon)\right)  \left(\sigma_\gamma\cdot e\mathfrak{c}\right) .
\end{align}

\noindent By Proposition \ref{jan28-2015-1038am}, the above residue is the product of the constant term of $E_{2}(\epsilon)$ at the cusp $w_{N_r/t}( \mathfrak{c})$ and the width of this cusp. In Subsection \ref{feb1-2014-357pm} it was shown that the constant term of $E_{2}(\epsilon)$ at the cusp $w_{N_r/t}(\mathfrak{c})$ for $\mathfrak{c}\in C_r^0$ is zero unless $\mathfrak{c}$ is of the form $$\mathfrak{c}_{r,d_t,d_Q}^{x,y} ~:=~  \left(\begin{bmatrix}y \\ d_td_QP x\end{bmatrix}_{N} ,  \begin{bmatrix}0 \\ \omega(d_td_QP x) \end{bmatrix}_{p^r}\right)^\prime$$ for some tuple $(d_t,d_Q,x,y)\in\mathscr{S}_t$. Therefore, (\ref{sep29-2014-950am}) can be written as
\begin{align}\label{feb3-2015-1048am}
\frac{\psi_0(p)^{-r}}{tp^{r-1}}\sum_{\substack{\epsilon\in \widehat{U_1/U_r}\\\gamma \in U_1/U_r\\(d_t,d_Q,x,y)\in \mathscr{S}_t}}\epsilon^{-1}(\gamma)\,\mathrm{Res}_{w_{N_r/t}\left(\mathfrak{c}_{r,d_t,d_Q}^{x,y}\right)}\!\big(E_{2}(\epsilon)\big) \left(\sigma_\gamma\cdot e\mathfrak{c}_{r,d_t,d_Q}^{x,y}\right).
\end{align}

\noindent By Lemma \ref{feb3-2014-446pm}, the constant term of $E_2(\epsilon)$ at the cusp $w_{N_r/t}\big(\mathfrak{c}_{r,d_t,d_Q}^{x,y}\big)$ is $$C\cdot\epsilon\left(\frac{f_\theta }{f_{\xi}d_tx}\right)\psi_0\!\left(\frac{yQp^r}{d_Q}\right)\theta_0^{-1}\!(d_tx)\,  \delta_{\theta,\psi}(\epsilon(u) - 1) \left(\prod_{\substack{\ell\,\mid\, f_\theta f_\psi \\ \ell\,\nmid\, f_{\xi}}}(1-(\xi\epsilon^{-1})(\ell)\ell^{-2})\right)L_p(-1,\xi_2^{-1}\epsilon)$$

\noindent where $C$ is a $p$-adic unit depending only on $\theta_0$ and $\psi_0$. Furthermore, in subsection \ref{jul3-2013-1106am} we showed that the width of the cusp $w_{N_r/t}\big(\mathfrak{c}_{r,d_t,d_Q}^{x,y}\big)$ is $td_QP/\gcd(y,t)$. Therefore,  (\ref{feb3-2015-1048am}) can be written as
\begin{multline}\label{feb3-2015-1124am}
CP\!\sum_{(d_t,d_Q,x,y)\in\mathscr{S}_t}\frac{d_Q}{\gcd(y,t)}\,\psi_0\!\left(\frac{yQ}{d_Q}\right) \theta_0^{-1}(d_tx) \\
\cdot \left(\frac{1}{p^{r-1}}\sum_{\substack{\epsilon\in \widehat{U_1/U_r}\\\gamma \in U_1/U_r}}\epsilon^{-1}(\gamma)\,\epsilon\left(-\frac{f_\theta }{f_{\xi}d_tx}\right)\mathcal{L}_\epsilon \left(\sigma_\gamma\cdot e\mathfrak{c}_{r,d_t,d_Q}^{x,y}\right)\right).
\end{multline}

\noindent where $$\mathcal{L}_\epsilon ~:=~ \delta_{\theta,\psi}(\epsilon(u) - 1) \left(\prod_{\substack{\ell\,\mid\, f_\theta f_\psi \\ \ell\,\nmid\, f_{\xi}}}(1-(\xi\epsilon^{-1})(\ell)\ell^{-2})\right)L_p(-1,\xi_2^{-1}\epsilon).$$

For each tuple $(d_t,d_Q,x,y)\in\mathscr{S}_t$ define $$\mathfrak{e}_{\infty,d_t,d_Q}^{x,y} ~=~\varprojlim_r e\mathfrak{c}_{r,d_t,d_Q}^{x,y}~\in~ C_{\Lambda}.$$ 

\noindent If $\lambda\in \Lambda_\infty$, we have $$\lambda \cdot \mathfrak{e}_{\infty,d_t,d_y}^{x,y} ~=~ \varprojlim_{r\geq 1}\left(\frac{1}{p^{r-1}}\sum_{\substack{\epsilon\in \widehat{U_1/U_r}\\\gamma \in U_1/U_r}}\epsilon^{-1}(\gamma)\lambda(\epsilon(u)-1) \left(\sigma_\gamma\cdot e\mathfrak{c}_{r,d_t,d_Q}^{x,y}\right)\right).$$

\noindent (cf. e.g. proof of \cite[Lemma 2.4.2]{OhtaES}). Noting that $$
A_{\theta,\psi}(\epsilon(u)-1) ~=~ \mathcal{L}_\epsilon\hspace{.3in}\text{and}\hspace{.3in}
\epsilon(u)^{s(-f_\theta/f_\xi d_tx)} ~=~ \epsilon\left(-\frac{f_\theta}{f_\xi d_tx}\right)$$

\noindent for all $\epsilon\in \widehat{U}_{1}$, we have $$(1+X)^{s(-f_\theta/f_\xi d_tx)}  \cdot A_{\theta,\psi}\cdot \mathfrak{e}_{\infty,d_t,d_Q}^{x,y}$$$$=\varprojlim_r \left(\frac{1}{p^{r-1}}\sum_{\substack{\epsilon\in \widehat{U_1/U_r}\\\gamma \in U_1/U_r}}\epsilon^{-1}(\gamma)\,\epsilon\left(-\frac{f_\theta}{f_\xi d_tx}\right)\mathcal{L}_\epsilon \left(\sigma_\gamma\cdot e\mathfrak{c}_{r,d_t,d_Q}^{x,y}\right)\right). $$

\noindent Putting this together with (\ref{feb3-2015-1124am}) we see that $$\mathrm{Res}_\Lambda(\mathcal{E}_{\theta_0,\psi_0;t}) ~=~ A_{\theta,\psi}\cdot \mathfrak{e}_{\theta_0,\psi_0;t}$$ where $$\mathfrak{e}_{\theta_0,\psi_0;t} :=CP\sum_{(d_t,d_Q,x,y)\in\mathscr{S}_t}\hspace{-.2in}\frac{d_Q\cdot(1+X)^{s(-f_\theta/f_\xi d_t x)}}{\gcd(y,t)}\,\,\psi_0\!\left(\frac{yQ}{d_Q}\right)\, \theta_0^{-1}(d_tx)\cdot\mathfrak{e}_{\infty,d_t,d_Q}^{x,y}.$$

\noindent All that remains to be shown is that $\mathfrak{e}_{\theta_0,\psi_0;t}\not\in \mathfrak{m}C_\Lambda$. To see this, note that the cusps $\mathfrak{e}_{\infty,d_t,d_Q}^{x,y}$ are distinct for all tuples $(d_t,d_Q,x,y)\in\mathscr{S}_t$ and $$CP\cdot \frac{d_Q\cdot(1+X)^{s(-f_\theta/f_\xi d_t x)}}{\gcd(y,t)}\,\,\psi_0\!\left(\frac{yQ}{d_Q}\right)\, \theta_0^{-1}(d_tx)~\in~\Lambda^\times.$$

\subsubsection{Residue of Eisenstein series associated to imprimitive characters}\label{jul3-2013-1109am}

Using Propositions $\ref{feb3-2014-705pm}$ and \ref{dec18-2013-1155am}, we can determine the image of $\Lambda$-adic Eisenstein series associated to imprimitive characters under the $\Lambda$-adic residue map.

\begin{theorem}\label{jun22-2015-947am} Suppose $\mathcal{E}_{\theta,\psi;t}\in M_\Lambda$. Then $$\mathrm{Res}_\Lambda(\mathcal{E}_{\theta,\psi;t}) ~=~ A_{\theta,\psi}\cdot \mathfrak{e}_{\theta,\psi;t}$$ where $$ \mathfrak{e}_{\theta,\psi;t} ~:=~\sum_{\substack{\alpha \mid D_\theta \\ \beta \mid D_\psi}}\alpha\mu(\alpha)\mu(\beta)\theta_0(\alpha)\psi_0(\beta)(1+X)^{s(\alpha)}\cdot \mathfrak{e}_{\theta_0,\psi_0;\alpha\beta t}~\in~C_\Lambda.$$ Furthermore, $\mathfrak{e}_{\theta,\psi;t}\not\in \mathfrak{m} C_\Lambda$.

\end{theorem}

\begin{proof} That the image of $\mathcal{E}_{\theta,\psi;t}$ under the $\Lambda$-adic residue map is given by the above sum is an immediate consequence of Propositions $\ref{feb3-2014-705pm}$ and \ref{dec18-2013-1155am}. To show $\mathfrak{e}_{\theta,\psi;t}\not\in \mathfrak{m} C_\Lambda$ we first note that,$$\mathfrak{e}_{\theta,\psi;t} ~:=~ CP\hspace{-.3in}\sum_{\substack{\alpha \mid D_\theta \\ \beta \mid D_\psi \\ (d_t,d_Q,x,y)\in\mathscr{S}_{\alpha\beta t}}}\!\!\!\!\!\frac{\alpha d_Q\mu(\alpha)\mu(\beta)(1+X)^{s(-\alpha f_\theta/f_\xi d_t x)}}{\gcd(y,\alpha\beta t)}\,\psi_0\!\left(\frac{yQ}{d_Q}\right) \theta_0^{-1}(d_tx)\cdot\mathfrak{e}_{\infty,d_t,d_Q}^{x,y}$$ 

\noindent (despite the notation, the integer $Q$ is dependent on $\alpha$ and $\beta$).  Recall that the cusps $\mathfrak{e}_{\infty,d_t,d_Q}^{x,y}$ are distinct for all tuples $(d_t,d_Q,x,y)\in \mathscr{S}_t$. Let $D_0$ denote the largest factor of $D_\theta D_\psi t$ that is prime to $f_\theta p$. Then the coefficient of $\mathfrak{e}_{\infty,D_0,1}^{1,1}$ in the above sum is given by $$ CP D_\theta \mu(D_\theta)\mu(D_\psi)\cdot \psi_0(Q)\theta_0^{-1}(D_0) \cdot (1+X)^{s(-D_\theta f_\theta/f_\xi D_0)} ~\in~ \Lambda^\times.$$\end{proof}

\section{Hida duality}\label{mar4-2015-911pm}\label{secDUAL}

The main result of this section will be a refinement of Hida's duality theorem. In order to state this refinement we must first use the results of the previous section to construct elements of $M_\Lambda$ arising from congruences between $\Lambda$-adic Eisenstein series and ordinary $\Lambda$-adic cusp forms.

\begin{theorem}[\cite{HidaBB}, \S7.3 Theorem 5]\label{feb6-2015-1235pm} We have a perfect pairing of $\Lambda$-modules $$S_\Lambda\times \mathfrak{h}_\Lambda \rightarrow \Lambda : (F,H)\mapsto a_1(F|H)$$
\end{theorem}

\begin{prop}\label{mar3-2015-406pm} Suppose $\mathcal{E}_{\theta,\psi;t}\in M_\Lambda$. Then ${\Lambda_{\infty}}(\mathfrak{e}_{\theta,\psi;t})$ is a free ${\Lambda_{\infty}}$-module. \end{prop}

\begin{proof} If $\lambda \cdot \mathfrak{e}_{\theta,\psi;t}=0$ for some $\lambda\in \Lambda_\infty$, we would have $\lambda\cdot \mathcal{E}_{\theta,\psi;t}\in S_{\Lambda_\infty}$, a contradiction.\end{proof}

\begin{prop}[\cite{OhtaEC2}, Lemma 2.1.1]\label{mar3-2015-407pm} The $\mathbb{Z}_p\llbracket X\rrbracket$-algebra ${\Lambda_{\infty}}$ is faithfully flat. \end{prop}

Using Propositions \ref{mar3-2015-406pm} and \ref{mar3-2015-407pm} along with Corollary \ref{mar3-2015-332pm}, we see that there exists an $F\in M_\Lambda$ mapping to $\mathfrak{e}_{\theta,\psi;t}$ under the $\Lambda$-adic residue map. We will now construct a canonical element of $M_\Lambda$ with this property. By Proposition \ref{dec10-209pm}, we know that $S_\Lambda$ is a free and finitely generated $\Lambda$-module. Let $\{F_1,\dots,F_s\}$ be a $\Lambda$-basis for $S_\Lambda$. By Theorem \ref{feb6-2015-1235pm}, we know that $\mathfrak{h}_\Lambda$ has a $\Lambda$-basis $\{B_1,\dots,B_s\}$ satisfying $$a_1(F_i|B_j)~=~\left\{\begin{array}{rcl}1 & & i=j\\ 0 & & i\neq j.\end{array}  \right.$$ For each $i$, let $\mathfrak{B}_i$ be \emph{any} element of $\mathfrak{H}_\Lambda$ that projects to $B_i$ under the surjection $\mathfrak{H}_\Lambda\twoheadrightarrow \mathfrak{h}_\Lambda$.

\begin{definition} Let $F\in M_\Lambda$ be any element satisfying $\mathrm{Res}_\Lambda(F)=\mathfrak{e}_{\theta,\psi;t}$, and define $$\mathcal{F}_{\theta,\psi;t} ~=~ F - \sum_{i=1}^s a_1(F|\mathfrak{B}_i)F_i.$$ 

\end{definition}

\noindent Because any two elements of $M_\Lambda$ mapping to $\mathfrak{e}_{\theta,\psi;t}$ will differ by a cusp form, our definition of $\mathcal{F}_{\theta,\psi;t}$ is independent of our choice of $F$.

Next, we record several properties of the form $\mathcal{F}_{\theta,\psi;t}$. First we note that $\mathcal{F}_{\theta,\psi;t}\not\in \mathfrak{m}M_\Lambda$ by Theorem \ref{jun22-2015-947am}. Next, note that $\mathrm{Res}_\Lambda(A_{\theta,\psi}\mathcal{F}_{\theta,\psi;t} - \mathcal{E}_{\theta,\psi;t})=0,$ which implies \begin{align}\label{feb7-2015-1116am}
\mathcal{F}_{\theta,\psi;t} &~=~ \frac{\mathcal{E}_{\theta,\psi;t} + \mathcal{G}_{\theta,\psi;t}}{A_{\theta,\psi}}
\end{align}
\noindent for a unique $\mathcal{G}_{\theta,\psi;t}\in S_\Lambda$. In fact, we can describe $\mathcal{G}_{\theta,\psi;t}$ explicitly. By construction we have $$a_1(\mathcal{F}_{\theta,\psi;t}|\mathfrak{B}_j) ~=~ a_1(F|\mathfrak{B}_j) - \sum_{i=1}^s a_1(F|\mathfrak{B}_i)a_1(F_i|\mathfrak{B}_j) ~=~ 0$$ for $1\leq j \leq s.$ Therefore, by (\ref{feb7-2015-1116am}) we have $$\mathcal{G}_{\theta,\psi;t} ~=~ \sum_{i=1}^m a_1(\mathcal{E}_{\theta,\psi;t}|\mathfrak{B}_i)F_i.$$

We can say even more about the Hecke action on $\mathcal{F}_{\theta,\psi;t}$. Since $\mathrm{Res}_\Lambda$ is a $\mathfrak{H}_\Lambda$-module homomorphism, (\ref{feb7-2015-1116am}) and Proposition \ref{sep8-2014-1136am} imply that for all primes $\ell\nmid N$, \begin{align*}
\mathcal{F}_{\theta,\psi;t}|T_\ell ~=~ \left\{\begin{array}{rcl}(\theta(\ell)\ell(1+X)^{s(\ell)} + \psi(\ell))\cdot\mathcal{F}_{\theta,\psi;t} + F_\ell & & \ell\nmid Np\\
\psi(\ell)\cdot \mathcal{F}_{\theta,\psi;t} + F_\ell & & \ell  = p\end{array}\right.,
\end{align*} where $F_\ell\in S_\Lambda$ and the subscript denotes the fact that this cusp form may depend on $\ell$.

With the forms $\mathcal{F}_{\theta,\psi;t}$ in hand, we are now ready to state the main result of this section. For any $\Lambda$-module $M$, we set $M_{Q(\Lambda)}=M\otimes_\Lambda Q(\Lambda)$ and denote the $\Lambda$-dual of $M$ by $\dual{M}$.

\begin{theorem}\label{jun22-2015-148pm} Let $\mathcal{V}$ be a free $\Lambda$-submodule of $M_\Lambda$ that contains $S_\Lambda$ and is stable under the action of $\mathfrak{H}_\Lambda$. Then $$\mathcal{V}_{Q(\Lambda)}~=~\langle \mathcal{E}_{\theta_1,\psi_1;t_1},\dots,\mathcal{E}_{\theta_m,\psi_m;t_m},F_1,\dots,F_s\rangle_{Q(\Lambda)}$$ where the Eisenstein  series $\mathcal{E}_{\theta_i,\psi_i;t_i}\in M_\Lambda$ are distinct. Define $\mathfrak{H}(\mathcal{V})$ to be the $\Lambda$-subalgebra of $\mathrm{End}_\Lambda(\mathcal{V})$ generated by the Hecke operators $\{T_n:n\geq 1\}$ and recall that
\begin{align*}
\mathcal{V}_0 ~:=~ \{F\in \mathcal{V}_{Q(\Lambda)} : a_n(F)\in \Lambda~\text{for all }n\geq1\}.
\end{align*}

\noindent If the following conditions are satisfied for all integers $i$ and $j$ with $1\leq i<j\leq m$:
\begin{enumerate}\itemsep0em\item[$(i)$] $(\theta_i)_0\nequiv (\theta_j)_0~(\mathrm{mod}\,\pi) \text{ or~~}(\psi_i)_0\nequiv (\psi_j)_0~(\mathrm{mod}\,\pi),$

\item[$(ii)$] $(\theta_i)_0\nequiv (\psi_j\omega^{-1})_0~(\mathrm{mod}\,\pi) \text{ or~~}(\psi_i)_0\nequiv (\theta_j\omega)_0~(\mathrm{mod}\,\pi),$\end{enumerate}

\noindent then we have $\mathcal{V}_0 = \langle \mathcal{F}_{\theta_1,\psi_1;t_1},\dots,\mathcal{F}_{\theta_m,\psi_m;t_m},F_1,\dots,F_s\rangle_{\Lambda} \subset M_\Lambda$ and the pairing $$\mathcal{V}_0\times \mathfrak{H}(\mathcal{V})\rightarrow \Lambda:(F,H)\mapsto a_1(F|H)$$ is perfect.

\end{theorem}

\begin{proof} That the pairing is perfect follows from a well known argument of Hida (\cite[\S7.3, Theorem 5]{HidaBB}), so we will restrict our efforts to proving the above characterization of $\mathcal{V}_0$. Clearly $$\langle \mathcal{F}_{\theta_1,\psi_1;t},\dots,\mathcal{F}_{\theta_m,\psi_m;t},F_1,\dots, F_s\rangle_{\Lambda}~\subset~ \mathcal{V}_0.$$

\noindent Suppose $F\in \mathcal{V}_0$. Then $F$ can be written as $$F ~=~ \sum_{i=1}^m\frac{P_i}{Q_i}\mathcal{F}_{\theta_i,\psi_i;t_i} + \sum_{i=1}^s \frac{P_i^\prime}{Q_i^\prime} F_i$$ for some $P_i,Q_i,P_i^\prime,Q_i^\prime\in \Lambda$ with $\gcd(P_i,Q_i)=1=\gcd(P_i^\prime,Q_i^\prime)$ and $Q_i,Q_i^\prime \neq 0$. Let $1\leq j \leq m$. By $(i)$, $(ii)$, and Lemma \ref{jun22-2015-1240pm}, for each $i\neq j$ with $1\leq i \leq m$, there exists a prime $\ell_i\nmid Np$ such that $$\left(\theta_j(\ell_i)\ell_i(1+X)^{s(\ell_i)} - \psi_j(\ell_i)\right) - \left(\theta_i(\ell_i)\ell_i(1+X)^{s(\ell_i)} - \psi_i(\ell_i)\right) ~\in~\Lambda^\times$$

\noindent Define $$H_{j} ~=~\prod_{\substack{i=1\\i\neq j}}^m T_{\ell_i} - \left(\theta_i(\ell_i)\ell_i(1+X)^{s(\ell_i)} - \psi_i(\ell_i)\right)~\in~ \bigcap_{\substack{i=1\\i\neq j}}^m \mathrm{Ann}_{\mathfrak{H}_\Lambda}(\mathcal{E}_{\theta_i,\psi_i;t_i}).$$ Then $\mathcal{F}_{\theta_i,\psi_i;t_i}|H_j \in S_\Lambda$ if $i\neq j$, while $\mathcal{F}_{\theta_j,\psi_j;t_j}|H_j = U_j \mathcal{F}_{\theta_j,\psi_j;t_j} + F_j$ for some $F_j \in S_\Lambda$ and $U_j\in \Lambda^\times$. Therefore, $$F|H_j ~=~ U_j\frac{P_j}{Q_j}\mathcal{F}_{\theta_j,\psi_j;t_j} + \sum_{i=1}^s \frac{P_i^{\prime\prime}}{Q_i^{\prime\prime}} F_i,$$

\noindent where $P_i^{\prime\prime},Q_{i}^{\prime\prime}\in\Lambda$ with $\gcd(P_i^{\prime\prime},Q_i^{\prime\prime})=1$ and $Q_i^{\prime\prime}\neq0$. Furthermore, since $\mathcal{V}_0$ is stable under the action of $\mathfrak{H}_\Lambda$, for $1\leq i\leq s$ we have $P_i^{\prime\prime}/Q_i^{\prime\prime}= a_1(F|H_j|\mathfrak{B}_i)\in\Lambda.$ This implies that $Q_j\mid a_n(\mathcal{F}_{\theta_j,\psi_j;t_j})$ for all $n\geq1$. Since $\mathcal{F}_{\theta_j,\psi_j;t_j}\not\in \mathfrak{m}M_\Lambda$, if $\psi_j\neq \mathbbm{1}$ it must be the case $Q_j\in  \Lambda^\times$. Suppose $\psi_j=\mathbbm{1}$ and $Q_j \not\in \Lambda^\times$. Then $a_0(\mathcal{F}_{\theta_j,\psi_j;t_j})\in\Lambda^\times$, while $a_n(\mathcal{F}_{\theta_j,\psi_j;t_j})\in \mathfrak{m}$ for all $n\geq 1$. Furthermore, we know that for all $n\geq 1$, $$a_{np}(\mathcal{F}_{\theta_j,\psi_j;t_j}) ~=~a_n(\mathcal{F}_{\theta_j,\psi_j;t_j}|T_p) ~=~ a_n(\mathcal{F}_{\theta_j,\psi_j;t_j}) + a_n(F_p)$$ for some $F_p\in S_\Lambda$. Since $a_n(\mathcal{F}_{\theta_j,\psi_j;t_j})\in \mathfrak{m}$ for all $n\geq 1$, it must be the case that $a_n(F_p)\in \mathfrak{m}$ for all $n\geq 1$ as well. We will now show that such a form cannot exist.

By the $q$-expansion map, we have an embedding, $$\bigoplus_{k=0}^\infty M_k(\Gamma_1)_{\mathcal{O}}\hookrightarrow \mathcal{O}\llbracket q\rrbracket,$$ \cite[\S1]{Hida86a}. Let us denote the image of this map by $M_\infty$. It is well known that there is a natural action of $T_p$ that preserves the space $M_\infty$ \cite[\S1]{Hida86a}. Specifically, if $f\in M_\infty$ with $$f ~=~ \sum_{k=0}^\infty f_k~\subset~ \mathcal{O}\llbracket q\rrbracket$$ where $f_k\in M_k(\Gamma_1)_{\mathcal{O}}$, then $$f|T_p ~:=~ \sum_{k=0}^\infty f_k|T_p~\subset~ \mathcal{O}\llbracket q\rrbracket.$$

Let 
\begin{align*}
f_2&~=~v_{2,\mathbbm{1}}(\mathcal{F}_{\theta_j,\psi_j;t_j})\in M_2(Np)_{\mathcal{O}}^\mathrm{ord} \\
f_0&~=~a_0(v_{2,\mathbbm{1}}(\mathcal{F}_{\theta_j,\psi_j;t_j}))\in M_0(Np)_{\mathcal{O}}.
\end{align*}

\noindent We know that $f_0 \equiv f_2~(\mathrm{mod}\,\pi).$ Furthermore, since the Hecke action commutes with the specialization map $v_{2,\mathbbm{1}}$, we know that $$f_2|T_p ~=~ v_{2,\mathbbm{1}}(\mathcal{F}_{\theta_j,\psi_j;t_j}|T_p) ~=~ f_2 + v_{2,\mathbbm{1}}(F_p)~  \equiv f_2~(\mathrm{mod}\,\pi).$$ Hence, $$ (p-1)f_0 ~\equiv_\pi~ pf_0 - f_2 ~\equiv_\pi~ f_0|T_p - f_2|T_p ~\equiv_\pi~ (f_0-f_2)|T_p~\equiv~0 ~(\mathrm{mod}\,\pi).$$ However, $(p-1) f_0\in\mathcal{O}^\times$, which gives us our contradiction. Hence, $Q_j\in \Lambda^\times$. Since $j$ was arbitrary, we have the result. \end{proof}

\subsection{Universal ordinary cuspidal Hecke algebra modulo Eisenstein ideal}

For this subsection we suppose $\mathcal{E}_{\theta,\psi}\in M_\Lambda$. Furthermore, we assume that $M_\theta M_\psi = N$ or $Np$, making $\mathcal{E}_{\theta,\psi}$ a normalized common eigenform for $\mathfrak{H}_\Lambda$.

\begin{prop}\label{sep10-2014-304pm} Let $I_{\theta,\psi}$ denote the image of $\mathrm{Ann}_{\mathfrak{H}_\Lambda}(\mathcal{E}_{\theta,\psi})$ in $\mathfrak{h}_\Lambda$. Then we have the following isomorphism of ${\Lambda}$-algebras $$\mathfrak{h}_\Lambda/I_{\theta,\psi}~\cong~ {\Lambda}/(A_{\theta,\psi}).$$ \end{prop}

\begin{proof} Let $\mathcal{G}_{\theta,\psi}$ be the cusp form associated to the Eisenstein series $\mathcal{E}_{\theta,\psi}$ by (\ref{feb7-2015-1116am}). Consider the map $\mathfrak{h}_\Lambda\rightarrow \Lambda/(A_{\theta,\psi})$ defined by \begin{align}\label{jun22-2015-205pm}H~\mapsto~ a_1(\mathcal{G}_{\theta,\psi}|H)~(\mathrm{mod}\,A_{\theta,\psi}).\end{align} Note that $a_1(\mathcal{G}_{\theta,\psi}|H)\equiv a_1(\mathcal{E}_{\theta,\psi}|\tilde{H})~(\mathrm{mod}\,A_{\theta,\psi})$ for any lift $\tilde{H}$ of $H$ to $\mathfrak{H}_\Lambda$. From this congruence and the fact that $\mathcal{E}_{\theta,\psi}$ is a normalized common eigenform for $\mathfrak{H}_\Lambda$, we know that the map (\ref{jun22-2015-205pm}) is a surjective $\Lambda$-algebra homomorphism. Furthermore, it is clear that $I_{\theta,\psi}$ is contained in the kernel of this map.  

Suppose $H$ lies in the kernel. Then $a_1(\mathcal{E}_{\theta,\psi}|\tilde{H})\in(A_{\theta,\psi})$, where $\tilde{H}$ is \emph{any} lift of $H$ to $\mathfrak{H}_\Lambda$. By Theorem \ref{jun22-2015-148pm} we know that there exists a Hecke operator $\mathfrak{B}_0\in\mathfrak{H}_\Lambda$ such that $a_1(\mathcal{F}_{\theta,\psi}|\mathfrak{B}_0)=1$ with $a_1(F|\mathfrak{B}_0)=0$ for all $F\in S_\Lambda$. This implies  $a_1(\mathcal{E}_{\theta,\psi}|\mathfrak{B}_0) = A_{\theta,\psi}$. Set $$\tilde{H}^\prime~=~ \tilde{H} - \frac{a_1(\mathcal{E}_{\theta,\psi}|\tilde{H})}{A_{\theta,\psi}}\mathfrak{B}_0.$$ Then by construction $\tilde{H}^\prime\in \mathrm{Ann}_{\mathfrak{H}_\Lambda}(\mathcal{E}_{\theta,\psi})$ and $\tilde{H}^\prime\mapsto H$ under the natural projection $\mathfrak{H}_\Lambda\twoheadrightarrow \mathfrak{h}_\Lambda$. 
\end{proof}

\section{Eichler-Shimura cohomology groups and the Iwasawa main conjecture}\label{secIWA}

Set $\mathcal{O} = \mathbb{Z}_p[\theta,\psi]$. For this section we will assume that $\mathcal{E}_{\theta,\psi}\in M_\Lambda$ with $M_\theta M_\psi = N$ or $Np$ and $A_{\theta,\psi}\not\in \Lambda^\times$. We assume the former so that $\mathcal{E}_{\theta,\psi}$ is a normalized common eigenform for $\mathfrak{H}_\Lambda$, while the latter is assumed to ensure $\mathfrak{h}_\Lambda \neq I_{\theta,\psi}$. In this section we will use the isomorphism $\mathfrak{h}_\Lambda/I_{\theta,\psi}\cong \Lambda/(A_{\theta,\psi})$ to extend Otha's proof of the Iwasawa main conjecture over $\mathbb{Q}$. We begin by introducing the $p$-adic Eichler-Shimura cohomology groups and their basic properties.  

Let $X_1(N_r)$ denote the canonical model of $\Gamma_r\setminus \mathbb{H}^*$ ($\mathbb{H}^*:=\mathbb{H}\cup\mathbb{Q}\cup\{\infty\}$) over $\mathbb{Q}$ in which the cusp at infinity is $\mathbb{Q}$-rational.

\begin{definition} The $p$-adic Eichler-Shimura cohomology group of level $N$ is defined to be $$\mathcal{T}~=~ \bigg(\varprojlim_r H^1_{\acute{\mathrm{e}}\mathrm{t}}(X_1(Np^r)\otimes_\mathbb{Q}\overline{\mathbb{Q}},\mathbb{Z}_p)^\mathrm{ord}\bigg) \widehat{\otimes}_{\mathbb{Z}_p} \mathcal{O}$$ where the projective limit is taken with respect to the trace mappings of \'{e}tale cohomology groups.

\end{definition}

\noindent There are natural actions of $G_\mathbb{Q}$ and $\mathfrak{h}^*_\Lambda$ on $\mathcal{T}$, and these actions commute with one another. Furthermore, $\mathcal{T}_{Q({\Lambda})}$ is a free $\mathfrak{h}^*_{Q({\Lambda})}$-module of rank 2, where $\mathfrak{h}_{Q(\Lambda)}^*:=\mathfrak{h}^*_\Lambda \otimes_\Lambda Q(\Lambda)$ \cite[Lemma 5.1.2]{OhtaEC1}. Therefore, we have a Galois representation $\rho:G_\mathbb{Q} \rightarrow \mathrm{GL}_2(\mathfrak{h}^*_{Q(\Lambda)}),$ and one can show that this representation satisfies the usual  Eichler-Shimura relations \cite[Theorem 5.1.5]{OhtaEC1}: If $\ell\nmid Np$ is a prime and $\Phi_\ell\in G_\mathbb{Q}$ is a geometric Frobenius at $\ell$, we have \begin{align*}\det(1-\rho(\Phi_\ell) X) ~=~ 1 - T_\ell^*X + \ell T_{\ell,\ell}^* X^2.\end{align*} 

Let $\chi_p:G_\mathbb{Q}\rightarrow \mathbb{Z}_p^\times$ denote the $p$-adic cyclotomic character. Then for all primes $\ell\nmid Np$, we have $\chi_p(\Phi_\ell) = \ell^{-1}$ while $\iota(\ell)$ acts on $\mathfrak{h}^*_\Lambda$ as multiplication by $T_{\ell,\ell}^*$. This implies that $\det\left(\rho(\Phi_\ell)\right)= \chi_{p}(\Phi_\ell)^{-1}\iota(\chi_{p}(\Phi_\ell))^{-1}$, which in turn implies $\det(\rho(\sigma)) =\chi_{p}(\sigma)^{-1}\iota(\chi_{p}(\sigma)^{-1})$ for all $\sigma\in G_\mathbb{Q}$ by the \v{C}ebotarev density theorem.

\subsection{The method of Kurihara and Harder-Pink}

In this section we employ the method of Kurihara \cite{Kuri} and Harder-Pink \cite{HP} to construct an abelian pro-$p$ extension $L/F_\infty$ from the representation $\rho$.

\begin{prop}[\cite{OhtaEC2}, Corollary 1.3.8] We have the following exact sequence of $\mathfrak{h}^*_\Lambda$-modules:
\begin{align*}
\begin{CD}0 @>>> \mathcal{T}_{+} @>>> \mathcal{T} @>>> \mathcal{T}/\mathcal{T}_{+} @>>> 0.\end{CD}
\end{align*} where $\mathcal{T}_+:= \mathcal{T}^{I_p}$. Furthermore, $\sigma\in I_p$ acts on $\mathcal{T}/\mathcal{T}_{+}$ by $\chi_p(\sigma)^{-1} \iota(\chi_p(\sigma))^{-1}.$

\end{prop}

Let $\sigma_0\in I_p$ be an element satisfying $\sigma_0(\zeta) = \zeta^{1+p}=\zeta^{u}$ for all primitive $p$-power roots of unity $\zeta\in\overline{\mathbb{Q}}$. Then the action of $\sigma_0$ on the quotient $\mathcal{T}/\mathcal{T}_{+}$ is given by $u^{-1}(1+X)^{-1}$. Set $S = u^{-1}(1+X)^{-1} - 1$ and define $$\mathcal{T}_- ~=~ \left\{x\in \mathcal{T}: \sigma_0\cdot x =(S+1)x\right\}.$$ 

\noindent Since the action of $G_{\mathbb{Q}_p}$ commutes with the action of $\mathfrak{h}^*_\Lambda$, we know that $\mathcal{T}_-$ is an $\mathfrak{h}^*_\Lambda$-module. For a ${\Lambda}$-module $M$, we set $M_S= M\otimes_{{\Lambda}}{\Lambda}[S^{-1}]$. One can show that $\mathcal{T}_S$ is a direct sum of the $\mathfrak{h}^*_{\Lambda,S}$-modules $\mathcal{T}_{-,S}$ and $\mathcal{T}_{+,S}$.

\begin{prop}[\cite{OhtaEC1}, Lemma 5.1.3.]\label{feb13-2015-1008am} $\mathcal{T}_{-,Q({\Lambda})}$, $\mathcal{T}_{+,Q({\Lambda})}$ are free $\mathfrak{h}^*_{Q({\Lambda})}$-modules of rank 1.

\end{prop}

\noindent Fixing $\mathfrak{h}^*_{Q({\Lambda})}$-bases for $\mathcal{T}_{-,Q({\Lambda})}$ and $\mathcal{T}_{+,Q({\Lambda})}$ (in that order), we write $${\rho(\sigma) ~=~ \begin{pmatrix}a(\sigma) & b(\sigma) \\ c(\sigma) & d(\sigma)\end{pmatrix}}.$$

Let $\mathcal{B}$ and $\mathcal{C}$ denote the $\mathfrak{h}^*_{\Lambda,S}$-submodules of $\mathfrak{h}_{Q({\Lambda})}^*$ generated by the sets $\{b(\sigma):\sigma \in G_\mathbb{Q}\}$ and $\{c(\sigma):\sigma \in G_\mathbb{Q}\}$, respectively.

\begin{prop}[\cite{OhtaCM}, Lemma 3.3.6.]\label{apr17-2014-1057am} $\mathcal{B}$ and $\mathcal{C}$ are faithful $\mathfrak{h}^*_{\Lambda,S}$-modules.

\end{prop}

Let $\mathcal{I}^*$ denote the image of  $\mathcal{I}:=\mathcal{I}_{\theta,\psi}:= \mathrm{Ann}_{\mathfrak{H}_\Lambda}(\mathcal{E}_{\theta,\psi})$ under the natural isomorphism induced by $H\mapsto H^*$. For later reference, we note that $\mathcal{I}^*$ is the ideal of $\mathfrak{H}^*_\Lambda$ generated by
\begin{align}
T_{d,d}^* - (\theta\psi)(d)(1+X)^{s(d)} & & \text{integers }d>0\text{ prime to }Np\label{jul2-2015-926am}\\
T_{\ell}^* - \theta(\ell)\ell(1+X)^{s(\ell)} - \psi(\ell) & & \text{primes }\ell\neq p\\
T_p^* - \psi(p).
\end{align}

\noindent Let  $I^*$ denote the image of $\mathcal{I}^*$ in $\mathfrak{h}^*_\Lambda$ and define the map $\tilde{\rho}$ by
\begin{center}
$\sigma\in G_\mathbb{Q} ~\mapsto~ \begin{pmatrix}\overline{a(\sigma)} & \overline{b(\sigma)}\\ 0 & \overline{d(\sigma)} \end{pmatrix}$,
\end{center}

\noindent where the bar indicates reduction modulo $I_{S}^*$. As the next proposition shows, this map is a Galois representation.

\begin{prop}[\cite{OhtaEC2}, Lemma 3.3.5]\label{apr7-124pm-2014} For any $\sigma,\tau \in G_\mathbb{Q}$ we have $a(\sigma)$, $d(\sigma)$, $b(\sigma)c(\tau)\in \mathfrak{h}^*_{\Lambda,S}$ with \begin{align*}
a(\sigma) &~\equiv~\theta_{p-2}(\sigma)^{-1}[\chi_{p}(\sigma)]^{-1}\iota([\chi_{p}(\sigma)])^{-1}~(\mathrm{mod}\,I_S^*)\\
d(\sigma)&~\equiv~\psi_0(\sigma)^{-1}~(\mathrm{mod}\,I_S^*)\\
b(\sigma)c(\tau)&~\equiv~ 0~(\mathrm{mod}\,I_S^*).
\end{align*}
\end{prop}

\noindent We now use the representation $\tilde{\rho}$ to construct our abelian pro-$p$ extension of $F_\infty$. Set
\begin{align*}
F_0 &~:=~ \text{the field corresponding to }\{\sigma\in G_\mathbb{Q} : \overline{a(\sigma)} \equiv 1 \equiv \overline{d(\sigma)}\}\\
L_0 &~:=~ \text{the field corresponding to}\ \ker(\tilde{\rho}),\\
L &~:=~ L_0F_\infty
\end{align*}

\noindent Then we have an injection of abelian groups $$\mathrm{Gal}(L_0/F_0)~\hookrightarrow~ \mathcal{B}/I_{S}^*\mathcal{B}~:~\sigma\mapsto \overline{b(\sigma)}.$$

\noindent Clearly $L_0/F_0$ is abelian, and the fact that $\mathcal{B}/I_S^*\mathcal{B}$ is a finitely generated $\Lambda$-module implies that $\mathrm{Gal}(L_0/F_0)$ is pro-$p$. By Proposition \ref{apr7-124pm-2014} we see that  $F\subseteq F_0\subseteq F_\infty$. Considering the definition of $\tilde{\rho}$, we see that $L_0/F_0$ is unramified at $p$. However, we know $F_\infty/F$ is totally ramified at $p$, which implies $L_0\cap F_\infty = F_0$. Therefore, $\mathrm{Gal}(L_0/F_0)=\mathrm{Gal}(L_0/L_0\cap F_\infty) \cong\mathrm{Gal}(L/F_\infty),$ and we see that $L/F_\infty$ is an abelian pro-$p$ extension.

\subsection{An isomorphism of the Iwasawa modules}\label{feb21-2015-1201pm}

The isomorphism between $\mathrm{Gal}(L/F_\infty)$ and $\mathrm{Gal}(L_0/F_0)$ implies that we have an injection \begin{align}\label{may20-2014-1034am}
\displaystyle{  \mathrm{Gal}(L/F_\infty)\hookrightarrow~ \mathcal{B}/I_{S}^*\mathcal{B}}.
\end{align}

\noindent In this subsection, we will show this injection induces an isomorphism of Iwasawa modules.

Recall that $\mathrm{Gal}(F_\infty/\mathbb{Q})\cong\Delta \times \Gamma$ acts on $\mathrm{Gal}(L/F_\infty)$ by conjugation, and the fact that the extension $L/F_\infty$ is abelian and pro-$p$ implies $\mathrm{Gal}(L/F_\infty)$ is a module over the Iwasawa algebra $\mathbb{Z}_p[\Delta]\llbracket \Gamma\rrbracket$. We want to identify $\mathbb{Z}_p[\Delta]\llbracket \Gamma\rrbracket$ with $\mathbb{Z}_p[\Delta]\llbracket X\rrbracket$ in a particular way. Recall that there is a natural isomorphism $U_1\cong \mathbb{Z}_p\cong \Gamma =\mathrm{Gal}(F_\infty/F)$. Let $\gamma_0\in\Gamma$ correspond to $u\in U_1$. Then $\gamma_0$ is a topological generator of $\Gamma$ and $[\chi_p(\gamma_0)] = u$. We then identify $\mathbb{Z}_p[\Delta]\llbracket \Gamma\rrbracket$ with $\mathbb{Z}_p[\Delta]\llbracket X\rrbracket$ by $$\gamma_0 ~\mapsto~ [\chi_p(\gamma_0)] = u ~\mapsto~ 1+X.$$

\noindent With the above identificiation, one can show by direct computation that the $\mathbb{Z}_p[\Delta]\llbracket X\rrbracket$ action on $\mathrm{Gal}(L/F_\infty)$ commutes with the injection (\ref{may20-2014-1034am}) as follows \cite[\S 5.3]{OhtaEC1}: for $\sigma\in\mathrm{Gal}(L/F_\infty)$ and $\delta \in \Delta$ we have 
\begin{align*}
\delta\cdot \sigma &~\mapsto~ \xi_1(\delta)\cdot\overline{b(\sigma)}\\
 X\cdot \sigma &~\mapsto~S\cdot\overline{b(\sigma)}.
\end{align*}

\noindent Consequently, $\mathrm{Gal}(L/F_\infty)$ is a ${\Lambda_\xi}$-module on which $\Delta$ acts via $\xi_1$.

Let $(\mathcal{B}/I^*_{S}\mathcal{B})^{\dagger}$ denote the ${\Lambda}[X^{-1}]$-module obtained from $\mathcal{B}/I^*_{S}\mathcal{B}$ by twisting the ${\Lambda}[S^{-1}]$-module structure by the involutive $\mathcal{O}$-module automorphism of ${\Lambda}$ given by $X\mapsto S$ (i.e. $X$ acts on $(\mathcal{B}/I^*_{S}\mathcal{B})^{\dagger}$ as multiplication by $S$).

\begin{prop}[\cite{OhtaEC2}, Lemma 3.3.11.]\label{apr7-2014-418pm} The injection $($\ref{may20-2014-1034am}\,$)$ induces an isomorphism of ${\Lambda}[X^{-1}]$-modules $\mathrm{Gal}(L/F_\infty)\otimes_{{\Lambda_\xi}} {\Lambda}[X^{-1}]\cong (\mathcal{B}/I_{S}^*\mathcal{B})^\dagger$.

\end{prop}

\subsection{Ramification in $L/F_\infty$}\label{feb21-2015-1202pm}

In this subsection, we will use the $\Lambda_\xi$-module structure of $\mathrm{Gal}(L/F_\infty)$ to characterize the ramification occurring in $L/F_\infty$.

Let $\ell\neq p$ be an arbitrary prime.  It is well known that the prime $\ell$ will not split completely in the cyclotomic $\mathbb{Z}_p$-extension $F_\infty/F$. Hence, there are only finitely many primes $\mathfrak{l}_1,\dots,\mathfrak{l}_m$ of $F_\infty$ lying above $\ell$. Consider the subgroup $G_\ell\subset \mathrm{Gal}(L/F_\infty)$ generated by the inertia subgroups $I_{\mathfrak{l}_i}$ for $1\leq i \leq m$. Let us call the corresponding fixed field $K_\ell$. Then $K_\ell/F_\infty$ is the maximal subextension of $L/F_\infty$ in which all of the $\mathfrak{l}_i$ are unramified. Of primary interest to us will be the group $\mathrm{Gal}(L/K_\ell)=G_\ell$.

\begin{lemma}[\cite{OhtaCM}, Lemma A.2.1]\label{jul2-2014-307pm} The Galois group $\mathrm{Gal}(L/K_\ell)$ is a cyclic ${\Lambda_\xi}$-module annihilated by $b_\ell(X):=(1+X)^{s(\ell)}  - \xi^{-1}_1(\ell)\ell.$

\end{lemma}

\begin{lemma}\label{feb24-2015-257pm} If $\ell \nmid N$ or $\xi_2(\ell)$ is not a $p$-power root of unity, then $\ell$ is unramified in $L/F_\infty$.

\end{lemma} 

\begin{proof} Recall that $\tilde{\rho}$ is unramified outside of $Np$, and $\overline{b(\sigma)}=0$ for $\sigma\in I_p$. Therefore, the injectivity of the map $\mathrm{Gal}(L/F_\infty)\hookrightarrow \mathcal{B}/I_S^*\mathcal{B}$ implies that $L/F_\infty$ is unramified outside of $N$. On the other hand, if $\xi_2(\ell)$ is not a $p$-power root of unity, we have $$|1-\xi_1^{-1}(\ell)\ell|_p~=~|1-\xi_2^{-1}(\ell)[\ell]|_p~=~ 1,$$ which implies $b_\ell(X)\in\Lambda^\times.$  This in turn implies that $\ell$ is unramified in $L/F_\infty$, by Lemma \ref{jul2-2014-307pm}.\end{proof}

\subsection{The Iwasawa main conjecture and the characteristic ideal of $L/F_\infty$}\label{apr29-2015-849pm}

In order to prove the main conjecture and determine $\mathrm{Char}_{{\Lambda_\xi}}(\mathrm{Gal}(L/F_\infty))$, we will employ the theory of Fitting ideals. Let us quickly recall the definition and some of the basic properties of these ideals  \cite[Appendix]{MW}.

\begin{definition}
Let $R$ be a commutative ring and $M$ an $R$-module of finite presentation. Take any presentation of $M$, $$R^m \xrightarrow{~\varphi~} R^n \xrightarrow{~~~} M \xrightarrow{~~~} 0.$$ The $(0^\text{th})$ Fitting ideal $\mathrm{Fitt}_R(M)$ is defined to be the ideal of $R$ generated by all $n\times n$ minors of $\varphi$. This ideal is independent of the choice of presentation.

\end{definition}

\begin{prop}[\cite{MW}, Appendix]\label{apr9-2014-839am} For any finitely generated $R$-module $M$, the following hold:

\begin{enumerate}
\itemsep0em
\item[$(1)$] If $M\twoheadrightarrow M^\prime$ is a surjection of $R$-modules, then $\mathrm{Fitt}_R(M)\subseteq \mathrm{Fitt}_R(M^\prime)$.

\item[$(2)$] If $M$ is a faithful $R$-module, then $\mathrm{Fitt}_R(M)=0$.

\item[$(3)$] For any $R$-algebra $R^\prime$, we have $\mathrm{Fitt}_{R^\prime}(M\otimes_R R^\prime) =\mathrm{Fitt}_R(M)\cdot R^\prime.$ 

\item[$(4)$] If $M$ is a direct sum of cyclic $R$-modules, say $M=R/\mathfrak{a}_1 \times\cdots\times R/\mathfrak{a}_t$), then $\mathrm{Fitt}_R(M) = \mathfrak{a}_1\cdot\dots\cdot\mathfrak{a}_t.$

\end{enumerate}
\end{prop}

We are now ready to prove the Iwasawa main conjecture over $\mathbb{Q}$. The following proof is based on the method of Ohta \cite{OhtaEC2}.

\begin{theorem}[The Iwasawa main conjecture over $\mathbb{Q}$]\label{feb20-2015-1011am} We have the following equality of ideals $$\mathrm{Char}_{{\Lambda_\xi}}(X_{\infty,\xi_1}) ~=~ (F(X,\xi_2^{-1})).$$
\end{theorem}

\begin{proof}  By a well known consequence of the analytic class number formula, it suffices to show $$\mathrm{Char}_{\Lambda_\xi}\!\left(X_{\infty,\xi_1}\right)~\subseteq~ (F(X,\xi_2^{-1}))$$ \cite[p.\,207]{MW}. In order to make the proof of this inclusion more manageable, we will make two claims from which the above inclusion follows easily. Once this is done, we will go about proving these claims.

Let $L^\mathrm{un}/F_\infty$ be the maximal unramified subextension of $L/F_\infty$, and define
\begin{align*}
\tilde{A} &~=~ \left(\prod_{\substack{\ell\mid f_\theta f_\psi \\ \ell\nmid f_\xi}}b_\ell(X)\right)\cdot F(X,\xi_2^{-1}).
\end{align*} 

\noindent Note that $\tilde{A}$ is a unit multiple of the image of $A=A_{\theta,\psi}$ under the isomorphism induced by $X\mapsto S$.\begin{enumerate}[itemindent=5em,labelsep=.5em,topsep=.75em,label={\bf Claim \arabic*:}]
\item $\displaystyle{\left(\prod_{\substack{\ell\mid N\\ \ell\nmid f_\xi}}b_\ell(X)\right)\cdot \mathrm{Char}_{\Lambda_\xi}(\mathrm{Gal}(L^\mathrm{un}/F_\infty))\subseteq  \mathrm{Char}_{\Lambda_\xi}(\mathrm{Gal}(L/F_\infty))  .}$

\item $X^m\cdot\mathrm{Char}_{\Lambda_\xi}(\mathrm{Gal}(L/F_\infty))\subseteq (\tilde{A})~\text{for some integer }m\geq 0$.
\end{enumerate}

\noindent Putting these two claims together, we get the following inclusion 
\begin{align}\label{apr29-2015-933pm}X^m\cdot \left(\prod_{\substack{\ell\mid N\\\ell\nmid f_\theta f_\psi}}b_\ell(X)\right) \mathrm{Char}_{\Lambda_\xi}(\mathrm{Gal}(L^\mathrm{un}/F_\infty))~\subseteq~ (F(X,\xi_2^{-1})).\end{align}

\noindent We know that $L^\mathrm{un}/F_\infty$ is an unramified pro-$p$ abelian extension on which $\Delta$ acts via $\xi_1$, which implies $\mathrm{Gal}(L^\mathrm{un}/F_\infty)$ is a quotient of $X_{\infty,\xi_1}$. Therefore, we have the following inclusion of characteristic ideals \begin{align*}\mathrm{Char}_{{\Lambda_\xi}}\!\left(X_{\infty,\xi_1}\right) ~\subseteq~ \mathrm{Char}_{\Lambda_\xi}(\mathrm{Gal}(L^\mathrm{un}/F_\infty)).\end{align*} \cite[Proposition 15.22]{Wash}. Putting this together with (\ref{apr29-2015-933pm}) we get \begin{align*}X^m\cdot \left(\prod_{\substack{\ell\mid N\\\ell\nmid f_\theta f_\psi}}b_\ell(X)\right) \mathrm{Char}_{\Lambda_\xi}\!\left(X_{\infty,\xi_1}\right)~\subseteq~ (F(X,\xi_2^{-1})).\end{align*}

\noindent A well-known result of Ferrero-Greenberg tells us that the power of $X$ dividing the generator of $\mathrm{Char}_{\Lambda_\xi}(X_{\infty,\xi_1})$ is equal to that dividing $F(X,\xi_2^{-1})$ \cite[\S 4]{FG}. In fact, this power is $1$ precisely when the pair $(\theta_0,\psi_0)$ is exceptional. Hence, the above inclusion holds with $m=0$.

It will now suffice to show $\gcd(b_\ell(X),F(X,\xi_2^{-1}))=1$. Recall from the proof of Lemma \ref{feb24-2015-257pm} that $b_\ell(X)$ is a unit if $\xi_2^{-1}(\ell)$ is not a $p$-power root of unity. Suppose $\xi_2^{-1}(\ell)$ is a $p$-power root of unity. Then any root of $b_\ell(X)$ must be of the form $u\zeta - 1$ where $\zeta$ satisfies $\zeta^{s(\ell)}=\xi_2^{-1}(\ell)$ (here we're using the fact that $u^{s(\ell)}=[\ell]$). Clearly $\zeta$ is a root of unity. By the same argument referenced above, we know that if $\zeta$ is not a $p$-power root of unity, then $u\zeta - 1$ is a unit, albeit possibly in some finite extension of $\mathcal{O}_{\xi}$. However, this would imply that the minimal polynomial of $u\zeta-1$ in $\mathcal{O}_{\xi}[X]$ is a unit in ${\Lambda_\xi}$. Thus, we may assume that $\zeta$ is a $p$-power root of unity. Evaluating $F(X,\xi_2^{-1})$ at $u\zeta - 1$ we get $L_p(1,\xi_2^{-1}\epsilon)$, where $\epsilon\in \widehat{U_1}$ satisfies $\epsilon(u)=\zeta^{-1}$. However, it is well known that $L_p(1,\xi_2^{-1}\epsilon)\neq 0$ \cite[\S 5.5]{Wash}.

\medskip

\emph{Proof of Claim 1:} 
Let $\ell$ be a prime dividing $N$ that does not divide $f_\xi$. By Lemma \ref{feb24-2015-257pm}, this is a necessary condition for the prime $\ell$ to ramify in $L/F_\infty$. Consider the following exact sequence of ${\Lambda_\xi}$-modules, $$0 \rightarrow \mathrm{Gal}(L/K_\ell)\rightarrow \mathrm{Gal}(L/F_\infty)\rightarrow  \mathrm{Gal}(K_\ell/F_\infty)\rightarrow 0.$$ By Lemma \ref{jul2-2014-307pm}, we know that $$ b_\ell(X)\cdot \text{Char}_{{\Lambda_\xi}}(\mathrm{Gal}(K_\ell/F_\infty)) ~\subseteq~ \text{Char}_{{\Lambda_\xi}}(\mathrm{Gal}(L/F_\infty)) .$$

\noindent Now, suppose $\ell^\prime\neq \ell$ is another prime dividing $N$ that does not divide $f_\xi$. Then we have the exact sequence $$0 \rightarrow \mathrm{Gal}(K_\ell/(K_\ell \cap K_{\ell^\prime}))\rightarrow \mathrm{Gal}(K_\ell/F_\infty)\rightarrow  \mathrm{Gal}((K_\ell \cap K_{\ell^\prime})/F_\infty)\rightarrow 0.$$ Since $\mathrm{Gal}(K_\ell/(K_\ell \cap K_{\ell^\prime}))\cong \mathrm{Gal}(K_\ell K_{\ell^\prime}/K_{\ell^\prime})$ with the latter being a quotient of $\mathrm{Gal}(L/K_{\ell^\prime})$, Lemma \ref{jul2-2014-307pm} tells us $$ b_{\ell^\prime}(X)\cdot \text{Char}_{{\Lambda_\xi}}(\mathrm{Gal}((K_\ell\cap K_{\ell^\prime})/F_\infty)) ~\subseteq~ \text{Char}_{{\Lambda_\xi}}(\mathrm{Gal}(K_\ell/F_\infty)) .$$

\noindent Letting $L^\text{ur}/F_\infty$ denote the maximal unramified subextension of $L/F_\infty$ and repeating the above argument, we get
\begin{align*}
\left(\prod_{\substack{\ell\mid N\\\ell\nmid f_\xi }}b_\ell(X)\right)\cdot \text{Char}_{{\Lambda_\xi}}(\mathrm{Gal}(L^\text{ur}/F_\infty))~\subseteq~ \text{Char}_{{\Lambda_\xi}}(\mathrm{Gal}(L/F_\infty)).
\end{align*}

\medskip

\emph{Proof of Claim 2:} By Proposition \ref{apr17-2014-1057am}, we know that $\mathcal{B}$ is a faithful $\mathfrak{h}^*_{\Lambda,S}$-module. Therefore, by Proposition \ref{apr9-2014-839am} (2) and (3), we have $\mathrm{Fitt}_{\mathfrak{h}^*_{\Lambda,S}/I^*_S}(\mathcal{B}/I^*_{S}\mathcal{B}) = 0$. We know that $\mathfrak{h}^*_{\Lambda,S}/I_{S}^*\cong {\Lambda}[S^{-1}]/(A)$ as ${\Lambda}[S^{-1}]$-modules by Proposition \ref{sep10-2014-304pm}, so applying Proposition \ref{apr9-2014-839am} (3) once more, we get
\begin{align*}
\mathrm{Fitt}_{{\Lambda}[S^{-1}]}(\mathcal{B}/I_{S}^*\mathcal{B}) ~\mathrm{mod}\,A ~=~ \mathrm{Fitt}_{\mathfrak{h}^*_{\Lambda,S}/I_{S}^*}(\mathcal{B}/I_{S}^*\mathcal{B}) ~=~ 0.
\end{align*}

\noindent This in turn implies $\mathrm{Fitt}_{{\Lambda}[X^{-1}]}((\mathcal{B}/I_{S}^*\mathcal{B})^\dagger)~\subseteq~(\tilde{A}).$ By the isomorphism of Proposition \ref{apr7-2014-418pm} we have $$\mathrm{Fitt}_{{\Lambda}[X^{-1}]}(\mathrm{Gal}(L/F_\infty)\otimes_{{\Lambda_\xi}}{\Lambda}[X^{-1}])~\subseteq~(\tilde{A}).$$ 

Now, we know that $\mathrm{Gal}(L/F_\infty)$ is a torsion ${\Lambda_\xi}$-module since we have an injection of ${\Lambda_\xi}$-modules $\mathrm{Gal}(L/F_\infty)\hookrightarrow \mathcal{B}/I_{S}^*\mathcal{B}$, and the latter is annihilated by $A\in {\Lambda_\xi}$. Suppose $\mathrm{Gal}(L/F_\infty)$ is pseudo-isomorphic to $\bigoplus_{i=1}^t{\Lambda_\xi}/(f_i)$. Tensoring with ${\Lambda}[X^{-1}]$ will kill any finite ${\Lambda_\xi}$-modules, so we have $$
\mathrm{Gal}(L/F_\infty)\otimes_{{\Lambda_\xi}}{\Lambda}[X^{-1}]~\cong~ \bigoplus_{i=1}^t{\Lambda}[X^{-1}]/(f_i).$$

\noindent Therefore, $$\mathrm{Char}_{{\Lambda_\xi}}(\mathrm{Gal}(L/F_\infty))\cdot {\Lambda}[X^{-1}] ~=~ \mathrm{Char}_{{\Lambda}[X^{-1}]}\big(\mathrm{Gal}(L/F_\infty)\otimes_{{\Lambda_\xi}}{\Lambda}[X^{-1}]\big)$$$$~=~ \left(\prod_{i=1}^tP_i(X)^{e_i}\right)\cdot{\Lambda}[X^{-1}]~=~ \mathrm{Fitt}_{{\Lambda}[X^{-1}]}(\mathrm{Gal}(L/F_\infty)\otimes_{{\Lambda_\xi}}{\Lambda}[X^{-1}]) ~\subseteq~ (\tilde{A}),$$ which implies $X^m\cdot \mathrm{Char}_{{\Lambda_\xi}}(\mathrm{Gal}(L/F_\infty))\subseteq(\tilde{A})$ for some integer $m\geq 0$.
\end{proof}

\begin{cor}\label{jun29-2015-313pm} Let $\tilde{A}_0=\tilde{A}/X$ if the pair $(\theta_0,\psi_0)$ is exceptional, with $\tilde{A}_0=\tilde{A}$ otherwise. Then $\mathrm{Char}_{{\Lambda_\xi}}(\mathrm{Gal}(L/F_\infty)) = (\tilde{A}_0)$.

\end{cor}

\begin{proof} By Theorem \ref{feb20-2015-1011am} and its proof, we have the following inclusion
\begin{align}\label{mar4-2015-1113am}\left(\prod_{\substack{\ell\mid N\\ \ell\nmid f_\theta f_\psi}}b_\ell(X)\right)\cdot (\tilde{A})  ~\subseteq~  \mathrm{Char}_{{\Lambda_\xi}}(\mathrm{Gal}(L/F_\infty)).\end{align} 

\noindent The fact that $\mathfrak{h}^*_{Q({\Lambda})}$ is a free and finitely generated $Q({\Lambda})$-module implies that $\mathcal{B}$ is a finitely generated $\mathfrak{h}_{\Lambda,S}^*$-module. Hence, we have a surjection $$\left({\Lambda}[S^{-1}]/(A)\right)^n~\cong~ \left(\mathfrak{h}^*_{\Lambda,S}/I_S^*\right)^n ~\twoheadrightarrow~\mathcal{B}/I_S^*\mathcal{B},$$ which implies $$(\tilde{A})^n~\subseteq~ \mathrm{Char}_{{\Lambda}[X^{-1}]}((\mathcal{B}/I_S^*\mathcal{B})^\dagger) ~\subseteq~ \mathrm{Char}_{{\Lambda_\xi}}(\mathrm{Gal}(L/F_\infty))\otimes_{\Lambda_\xi} {\Lambda}[X^{-1}].$$ From the injection $\mathrm{Gal}(L/F_\infty) \hookrightarrow (\mathcal{B}/I_S^*\mathcal{B})^\dagger,$ we see that there are no elements of $\mathrm{Gal}(L/F_\infty)$ annihilated by $X$. Therefore, $X\nmid \mathrm{Char}_{\Lambda_\xi}(\mathrm{Gal}(L/F_\infty))$ and the above inclusion implies \begin{align}\label{apr29-2015-1004pm}(\tilde{A})^n~\subseteq~\mathrm{Char}_{{\Lambda_\xi}}(\mathrm{Gal}(L/F_\infty)).\end{align}

\noindent In the remarks preceding the proof of Claim 1, it was shown that $b_\ell(X)$ and $F(X,\xi_2^{-1})$ are coprime. Therefore, by (\ref{mar4-2015-1113am}) and (\ref{apr29-2015-1004pm}) we have $(\tilde{A})\subseteq \mathrm{Char}_{{\Lambda_\xi}}(\mathrm{Gal}(L/F_\infty)).$ Combining this with the result of Ferrero-Greenberg and Claim 2 from the proof of Theorem \ref{feb20-2015-1011am} we obtain the desired result. \end{proof}

\bibliographystyle{plain}


\end{document}